\documentclass{amsart}
\usepackage{amsmath}
\usepackage{amscd}
\usepackage{amssymb}
\usepackage{amsfonts}
\newtheorem{theorem}{Theorem}[section]
\newtheorem{lemma}[theorem]{Lemma}
\newtheorem{corollary}[theorem]{Corollary}

\newtheorem{proposition}[theorem]{Proposition}

\theoremstyle{definition}

\theoremstyle{remark}
\newtheorem{remark}[theorem]{Remark}
\numberwithin{equation}{section}

\begin{document}
\title[First eigenvalue of the $p$-Laplace operator along the Ricci flow]
{First eigenvalue of the $p$-Laplace operator\\ along the Ricci
flow}

\author{Jia-Yong Wu}
\address{Department of Mathematics, East China Normal
University, Dong Chuan Road 500, Shanghai 200241, People's Republic
of China} \email{jywu81@yahoo.com}

\author{Er-Min Wang}
\address{Department of Mathematics, East China Normal
University, Dong Chuan Road 500, Shanghai 200241, People's Republic
of China} \email{wagermn@126.com}

\author{Yu Zheng}
\address{Department of Mathematics, East China Normal
University, Dong Chuan Road 500, Shanghai 200241, People's Republic
of China} \email{zhyu@math.ecnu.edu.cn}

\thanks{This work is partially supported by the
NSFC10871069.} \subjclass[2000]{Primary 58C40; Secondary 53C44.}
\date{July 1, 2009.}

\dedicatory{} \keywords{Ricci flow; first eigenvalue; $p$-Laplace
operator; continuity; monotonicity; differentiability.}

\begin{abstract}
In this paper, we mainly investigate continuity, monotonicity and
differentiability for the first eigenvalue of the $p$-Laplace
operator along the Ricci flow on closed manifolds. We show that the
first $p$-eigenvalue is strictly increasing and differentiable
almost everywhere along the Ricci flow under some curvature
assumptions. In particular, for an orientable closed surface, we
construct various monotonic quantities and prove that the first
$p$-eigenvalue is differentiable almost everywhere along the Ricci
flow without any curvature assumption, and therefore derive a
$p$-eigenvalue comparison-type theorem when its Euler characteristic
is negative.
\end{abstract}
\maketitle

\section{Introduction}\label{sec1}
Given a compact Riemannian manifold $(M^n,g_0)$ without boundary,
the Ricci flow is the following evolution equation
\begin{equation}\label{flow}
\frac{\partial}{\partial t}g_{ij}=-2R_{ij}
\end{equation}
with the initial condition $g(x,0)=g_0(x)$, where $R_{ij}$ denotes
the Ricci tensor of the metric $g(t)$. The normalized Ricci flow is
\begin{equation}\label{norflow}
\frac{\partial}{\partial
\tilde{t}}\tilde{g}_{ij}=-2\tilde{R}_{ij}+\frac 2 n \tilde{r}
\tilde{g}_{ij},
\end{equation}
where $\tilde{g}(\tilde{t}):=c(t)g(t)$,
$\tilde{t}(t):=\int^t_0c(\tau)d\tau$ and
\begin{equation}
\begin{aligned}\label{avesca}
c(t):=\exp\left(\frac 2n\int^t_0r(\tau)d\tau\right),\quad
\quad\tilde{r}:={\int_M \tilde{R}d\tilde{\mu}}\Big/{\int_M
d\tilde{\mu}},
\end{aligned}
\end{equation}
($d\tilde{\mu}$ and $\tilde{R}$ denote the volume form and the
scalar curvature of the metric $\tilde{g}(\tilde{t})$,
respectively.) which preserves the volume of the initial manifold.
Both evolution equations were introduced by R.S. Hamilton to
approach the geometrization conjecture in \cite{Hamilton}. Recently,
studying the eigenvalues of geometric operator is a very powerful
tool for understanding of Riemannian manifolds. In \cite{Perelman},
G. Perelman introduced the functional
\[
\mathcal {F}(g(t), f(t)):=\int_M \left(R+|\nabla
f|^2\right)e^{-f}d\mu
\]
and showed that this functional is nondecreasing along the Ricci
flow coupled to a backward heat-type equation. More precisely, if
$g(t)$ is a solution to the Ricci flow (\ref{flow}) and the coupled
$f(x,t)$ satisfies the following evolution equation:
\[
\frac{\partial f}{\partial t}=-\Delta f+|\nabla f|^2-R,
\]
then we have
\[
\frac{\partial \mathcal {F}}{\partial t}=2\int_M\left|Ric
+\nabla^2f\right|^2e^{-f}d\mu.
\]
If we define
\[
\lambda(g(t)):=\inf\limits_{f\neq 0}\left\{\mathcal{F}(g(t),
f(t)):f\in C^\infty(M), \int_M e^{-f}d\mu=1\right\},
\]
then $\lambda(g(t))$ is the lowest eigenvalue of the operator
$-4\Delta+R$, and the increasing of the functional $\mathcal {F}(g,
f)$ implies the increasing of $\lambda(g(t))$.

Later in \cite{Cao1}, X.-D. Cao studied the eigenvalues $\lambda$
and eigenfunctions $f$ of the new operator $-\Delta+R/2$ satisfying
$\int_M f^2 d\mu=1$ on closed manifolds with nonnegative curvature
operator. In fact he introduced
\begin{equation}\label{caooper1}
\lambda(f,t):=\int_M \left(-\Delta f+\frac R2 f\right)f d\mu,
\end{equation}
where $f$ is a smooth function satisfying $\int_M f^2 d\mu=1$ and
obtained the following

\vspace{0.5em}

\noindent \textbf{Theorem A.} (X.-D. Cao \cite{Cao1}) \emph{On a
closed Riemannian manifold with nonnegative curvature operator, the
eigenvalues of the operator $-\Delta+\frac R2$ are nondecreasing
under the unnormalized Ricci flow, i.e.
\begin{equation}\label{caomono1}
\frac{d}{dt}\lambda(f,t)=2\int_M Ric(\nabla f,\nabla f)+\int_M
|Ric|^2 f^2 d\mu\geq 0.
\end{equation}}
In (\ref{caomono1}), when $\frac{d}{dt}\lambda(f,t)$ is evaluated at
time $t$, $f$ is the corresponding eigenfunction of $\lambda(t)$.
Hence $\lambda(t)$ is nondecreasing.

\vspace{0.5em}

Shortly thereafter J.-F. Li in \cite{JFLi} dropped the curvature
assumption and also obtained the above result for the operator
$-\Delta+\frac R2$. In fact, he used new entropy functionals to
derive a general result.

\vspace{0.5em}

\noindent \textbf{Theorem B.} (J.-F. Li \cite{JFLi}) \emph{On a
compact Riemannian manifold $(M, g(t))$, where $g(t)$ satisfies the
unnormalized Ricci flow for $t\in[0,T)$, the lowest eigenvalue
$\lambda_k$ of the operator $-4\Delta+kR$ $(k>1)$ is nondecreasing
under the unnormalized Ricci flow. The monotonicity is strict unless
the metric is Ricci-flat.}

\vspace{0.5em}

At around the same time, X.-D. Cao in \cite{Cao} also considered the
general operator $-\Delta+cR$ $(c\geq 1/4)$, and derived the
following exact monotonicity formula.

\vspace{0.5em}

\noindent \textbf{Theorem C.} (X.-D. Cao \cite{Cao}) \emph{Let
$(M^n, g(t))$, $t\in[0,T)$, be a solution of the unnormalized Ricci
flow (\ref{flow}) on a closed manifold $M^n$. Assume that
$\lambda(t)$ is the lowest eigenvalue of $-\Delta+cR$ $(c\geq 1/4)$
and $f=f(x, t)>0$ satisfies
\[
-\Delta f(x,t)+c Rf(x,t)=\lambda(t)f(x,t)
\]
with $\int_M f^2d\mu=1$. Then under the unnormalized Ricci flow, we
have
\begin{equation}\label{caomono2}
\frac{d}{dt}\lambda(t)=\frac1 2\int_M |Ric+\nabla^2 \varphi|^2
e^{-\varphi}d\mu+\frac{4c-1}{2}\int_M |Ric|^2e^{-\varphi}d\mu\geq 0,
\end{equation}
where $e^{-\varphi}=f^2$.}

\vspace{0.5em}

On the other hand, L. Ma in \cite{Ma} considered the eigenvalues of
the Laplace operator along the Ricci flow and proved the following
result.

\vspace{0.5em}

\noindent \textbf{Theorem D.} (L. Ma \cite{Ma}) \emph{Let $g=g(t)$
be the evolving metric along the unnormalized Ricci flow with
$g(0)=g_0$ being the initial metric in $M$. Let $D$ be a smooth
bounded domain in $(M, g_0)$. Let $\lambda>0$ be the first
eigenvalue of the Laplace operator of the metric $g(t)$. If there is
a constant such that the scalar curvature $R\geq2a$ in
$D\times\{t\}$ and the Einstein tensor
\[
E_{ij}\geq-ag_{ij}\quad \quad\mathrm{in}\quad D\times\{t\},
\]
then we have $\lambda'\geq0$, that is, $\lambda$ is nondecreasing in
$t$, furthermore, $\lambda'(t)>0$ for the scalar curvature $R$ not
being the constant $2a$. The same monotonicity result is also true
for other eigenvalues.}

\vspace{0.5em}

Moreover S.-C. Chang and P. Lu in \cite{ChLu} studied the evolution
of Yamabe constant under the Ricci flow and gave a simple
application. Motivated by the above works, in this paper we will
study the first eigenvalue of the $p$-Laplace operator whose metric
satisfying the Ricci flow. For the $p$-Laplace operator, besides
many interesting properties between the eigenvalues of the
$p$-Laplace operator and geometrical invariants were pointed out in
fixed metrics (e.g. \cite{Grosjean}, \cite{Kawai}, \cite{Kotschwar},
\cite{Matei}), the first author in \cite{Wu2} studied the
monotonicity for the first eigenvalue of the $p$-Laplace operator
along the Ricci flow on closed manifolds.

In this paper, on one hand we will improve those results in
\cite{Wu2} and discuss the differentiability for the first
eigenvalue of the $p$-Laplace operator along the unnormalized Ricci
flow. Meanwhile we construct some monotonic quantities along the
unnormalized Ricci flow. On the other hand, we will deal with the
case of the normalized Ricci flow in the same way and give an
interesting application. For the unnormalized Ricci flow, we first
have
\begin{theorem}\label{T101a}
Let $g(t)$, $t\in[0,T)$, be a solution of the unnormalized Ricci
flow (\ref{flow}) on a closed manifold $M^n$ and $\lambda_{1,p}(t)$
be the first eigenvalue of the $p$-Laplace operator $(p>1)$ of
$g(t)$. If there exists a nonnegative constant $\epsilon$ such that
\begin{equation}\label{tiaojian1}
R_{ij}-\tfrac{R}{p}g_{ij}\geq -\epsilon g_{ij}\quad \quad
\mathrm{in}\quad M^n\times[0,T)
\end{equation}
and
\begin{equation}\label{tiaojian2}
R\geq p\cdot\epsilon\quad \mathrm{and}\quad R\not\equiv
p\cdot\epsilon\quad \quad\mathrm{in}\quad M^n\times\{0\},
\end{equation}
then $\lambda_{1,p}(t)$ is strictly increasing and differentiable
almost everywhere along the unnormalized Ricci flow on $[0,T)$.
\end{theorem}
\begin{remark}
(1). In \cite{Wu2}, the first author proved a similar result as in
Theorem \ref{T101a}, where he assumed $p\geq 2$, inequality
(\ref{tiaojian1}) and $R>p\cdot\epsilon$ in $M^n\times\{0\}$, which
are a little stronger than assumptions of Theorem \ref{T101a}. The
key difference is that the proof approach here is different from
that in \cite{Wu2}.

(2). As mentioned Remark 1.2 in \cite{Wu2}, the time interval
$[0,T)$ of Theorem \ref{T101a} here may be not the maximal time
interval of existence of the unnormalized Ricci flow. In fact if we
trace (\ref{tiaojian1}) and assume that $p<n$, then we have an upper
bound estimate for the scalar curvature $(\epsilon\neq 0)$. But as
we all known, curvature operator must be blow-up as $t\rightarrow T$
$(T<\infty)$ when the curvature operator is positive and $[0,T)$ is
the maximal time interval (see Theorem 14.1 in \cite{Hamilton}).

(3). Theorem \ref{T101a} still holds if the conditions
(\ref{tiaojian1}) and (\ref{tiaojian2}) are replaced by
$R_{ij}-\tfrac{R}{p}g_{ij}>-\epsilon g_{ij}$ in  $M^n\times[0,T)$
and $R\geq p\cdot\epsilon$ in $M^n\times\{0\}$.

(4). For any closed $2$-surface and $3$-manifold, we can relax the
above assumptions (\ref{tiaojian1}) and (\ref{tiaojian2}) to the
only initial curvature assumptions by the Hamilton's maximum
principle. We refer the reader to \cite{Wu2} for similar results.
\end{remark}

\begin{remark}
Most recently, in \cite{CHL} X.-D. Cao, S.-B. Hou and J. Ling
derived a monotonicity formula for the first eigenvalue of
$-\Delta+aR$ $(0<a\leq 1/2)$ on closed surfaces with nonnegative
scalar curvature under the Ricci flow. Meanwhile they obtained
various monotonicity formulae and estimates for the first eigenvalue
on closed surfaces.
\end{remark}

Furthermore, if less curvature assumptions are given, we can
construct two classes of monotonic (increasing and decreasing)
quantities about the first eigenvalue of the $p$-Laplace operator
along the unnormalized Ricci flow. We refer the reader to Section
\ref{sec3h} for the more detailed discussions (see Theorems
\ref{T10b} and \ref{norma1}, and Corollary \ref{T103aa}).

\vspace{0.5em}

For the normalized Ricci flow, unfortunately we may not get any
monotonicity for the first eigenvalue of the $p$-Laplace operator in
general. However, if we know the first $p$-eigenvalue
differentiability along the unnormalized Ricci flow, from the
relation to the unnormalized Ricci flow, we can give another way to
derive the first $p$-eigenvalue differentiability along the
normalized Ricci flow (see Theorem \ref{T100b} of Section
\ref{sec3b}).

Besides, the most important result is that we can construct various
monotonic quantities about the first eigenvalue of the $p$-Laplace
operator along the normalized Ricci flow on closed $2$-surfaces
without any curvature assumption. This also leads to the first
$p$-eigenvalue differentiability along the normalized Ricci flow on
closed $2$-surfaces without any curvature assumption.

\begin{theorem}\label{coro17}
Let $\tilde{g}(\tilde{t})$, $\tilde{t}\in[0,\infty)$, be a solution
of the normalized Ricci flow (\ref{norflow}) on a closed surface
$M^2$ and let $\lambda_{1,p}(\tilde{t})$ be the first eigenvalue of
the $p$-Laplace operator of the metric $\tilde{g}(\tilde{t})$. Then
each of the following quantities

\begin{enumerate}
  \item $\lambda_{1,p}(\tilde{t})\cdot\left(\frac{\rho_0}{\tilde{r}}
-\frac{\rho_0}{\tilde{r}}e^{\tilde{r}\tilde{t}}
+e^{\tilde{r}\tilde{t}}\right)^{p/2}$
\quad\quad\quad\quad\quad\quad$(p\geq2)$,\\
$\lambda_{1,p}(\tilde{t}){\cdot}\kern-3pt\left(\frac{\rho_0}{\tilde{r}}
{-}\frac{\rho_0}{\tilde{r}}e^{\tilde{r}\tilde{t}}
{+}e^{\tilde{r}\tilde{t}}\right){\cdot}\exp\left[\left(1{-}\frac
p2\right)\kern-2pt\frac{C}{\tilde{r}}e^{\tilde{r}\tilde{t}}\right]$\quad
$(1<p<2)$,\quad\,\,\, $\mathrm{if}$ $\chi(M^2)<0$;

  \item $\lambda_{1,p}(\tilde{t})\cdot\left(1+C\tilde{t}\right)^{p/2}$
\quad\quad\quad\quad\quad\quad\quad\quad\quad\quad $(p\geq2)$,\\
$\lambda_{1,p}(\tilde{t})\cdot\left(1+C\tilde{t}\right)\cdot
e^{\left(1{-}p/2\right)C\tilde{t}}$
\quad\quad\quad\quad\quad\quad\quad $(1<p<2)$, \quad\, $\mathrm{if}$
$\chi(M^2)=0$;

  \item $\ln\lambda_{1,p}(\tilde{t})+\frac p2\cdot\left(\frac {C}{\tilde{r}}
e^{\tilde{r}\tilde{t}}+\tilde{r}\tilde{t}\right)$
\quad\quad\quad\quad\quad\quad\quad $(p\geq2)$,\\
$\ln\lambda_{1,p}(\tilde{t})+\left(2-\frac p2\right)\frac
{C}{\tilde{r}}e^{\tilde{r}\tilde{t}}+\tilde{r}\tilde{t}$
\quad\quad\quad\quad\quad\quad $(1<p<2)$, \quad\,\, $\mathrm{if}$
$\chi(M^2)>0$
\end{enumerate}
is increasing and therefore $\lambda_{1,p}(\tilde{t})$ is
differentiable almost everywhere along the normalized Ricci flow on
$[0,\infty)$, where $\chi(M^2)$ denotes its Euler characteristic,
$\rho_0:=\inf_{M^2}R(0)$ and $C>0$ is a constant depending only on
the initial metric.
\end{theorem}

In the same way, we can also obtain the decreasing quantities on
closed $2$-surfaces.
\begin{theorem}\label{thm19}
Under the same assumptions as in Theorem \ref{coro17}, then each of
the following quantities
\begin{enumerate}
  \item $\ln\lambda_{1,p}(\tilde{t})-\frac p2\cdot\frac {C}{\tilde{r}}
e^{\tilde{r}\tilde{t}}$
\quad\quad\quad\quad\quad\quad\quad\quad\quad\quad\,$(p\geq2)$,\\
$\lambda_{1,p}(\tilde{t}){\cdot}\kern-3pt\left(\frac{\rho_0}{\tilde{r}}
{-}\frac{\rho_0}{\tilde{r}}e^{\tilde{r}\tilde{t}}
{+}e^{\tilde{r}\tilde{t}}\right)^{\kern-2pt(\frac p2-1)}
\kern-6pt{\cdot}
\exp\kern-2pt\left({-}\frac{C}{\tilde{r}}e^{\tilde{r}\tilde{t}}\right)$
\quad $(1<p<2)$,\quad\quad$\mathrm{if}$ $\chi(M^2)<0$;

  \item $\ln\lambda_{1,p}(\tilde{t})-\frac p2\cdot C\tilde{t}$
\quad\quad\quad\quad\quad\quad\quad\quad\quad\quad\quad$(p\geq2)$,\\
$\lambda_{1,p}(\tilde{t})\cdot\left(1+C\tilde{t}\right)^{(\frac
p2-1)}\cdot
e^{-C\tilde{t}}$\quad\quad\quad\quad\quad\quad\quad$(1<p<2)$
\quad\quad $\mathrm{if}$ $\chi(M^2)=0$;

  \item $\ln\lambda_{1,p}(\tilde{t})-\frac p2\cdot\frac {C}{\tilde{r}}
e^{\tilde{r}\tilde{t}}$
\quad\quad\quad\quad\quad\quad\quad\quad\quad\quad$(p\geq2)$,\\
$\ln\lambda_{1,p}(\tilde{t}){-}\left(2{-}\frac
p2\right)\frac{C}{\tilde{r}}{\cdot}e^{\tilde{r}\tilde{t}}{-}
\left(1{-}\frac p2\right)\tilde{r}\tilde{t}$\quad\quad\quad\quad
$(1<p<2)$\quad\quad\,\,$\mathrm{if}$ $\chi(M^2)>0$
\end{enumerate}
is decreasing and therefore $\lambda_{1,p}(\tilde{t})$ is
differentiable almost everywhere along the normalized Ricci flow on
$[0,\infty)$, where $\chi(M^2)$, $\rho_0$ and $C$ are as in Theorem
\ref{coro17}.
\end{theorem}

\begin{remark}
We may apply similar techniques above to obtain interesting
monotonic quantities about the first eigenvalue of the $p$-Laplace
operator along the normalized Ricci flow in high-dimensional cases
under some curvature assumptions, but the proof needs more
computing. Here we omit this aspect.
\end{remark}

Some parts of results for $p=2$ above were proved by L. Ma \cite{Ma}
and J. Ling \cite{Ling2}. But our method of proof is different from
theirs. Their proofs strongly depend on the differentiability for
the eigenvalues and the corresponding eigenfunctions. But in our
setting ($p\geq 2$) it is not clear whether the eigenvalue or the
corresponding eigenfunction is differentiable in advance. Our method
is similar to X.-D. Cao's trick in \cite{Cao1}, which does not
depend on the differentiability for the eigenvalues or the
corresponding eigenfunctions.

\vspace{0.5em}

With the help of Theorem \ref{coro17}, our below topic is to extend
an earlier J. Ling's result for $p=2$ (see \cite{Ling}). Here we
call it $p$-eigenvalue comparison-type theorem. For the convenience
of introducing our result, we shall state a well-known fact, which
was proved by R.S. Hamilton and B. Chow (see also \cite{Chow-Knopf},
chapter 5 for details).

\vspace{0.5em}

\textbf{Theorem E.} (Chow-Hamilton, \cite{Chowjdg} and
\cite{Hamilton2}) \emph{If $(M^2, g)$ is a closed surface, there
exists a unique solution $g(t)$ of the normalized Ricci flow
(\ref{norflow}). The solution exists for all the time. As
$t\rightarrow\infty$, the metrics $g(t)$ converge uniformly in any
$C^k$-norm to a smooth metric $\bar{g}(=g(\infty))$ of constant
curvature.}

\vspace{0.5em}

Let $(M^2,g)$ be a closed surface. Let $K_g$, $\kappa_g$, $\mathrm
{Area}_g(M^2)$ denote the Gauss curvature, the minimum of the Gauss
curvature, the area of the surface $M^2$, respectively.
$\lambda_{1,p}(g)$ denotes the first eigenvalue of the $p$-Laplace
operator with respect to the metric $g$. Then we prove that
\begin{theorem}\label{T201}
($p$-eigenvalue comparison-type theorem). Suppose that $(M^2,g)$ is
a closed surface with its Euler characteristic $\chi(M^2)<0$. The
Ricci flow with initial metric $g$ converges uniformly to a smooth
metric $\bar{g}$ of constant curvature. Then for any $p\geq 2$,
\begin{equation}
\frac{\lambda_{1,p}(\bar{g})}{\lambda_{1,p}(g)}\geq\left(\frac
{\kappa_{\bar{g}}}{\kappa_g}\right)^{p/2}
\end{equation}
and the constant Gauss curvature for metric $\bar{g}$ is
$\kappa_{\bar{g}}=2\pi \chi(M^2)/\mathrm {Area}_g(M^2)$.
\end{theorem}

In conclusion, our new contribution of this paper is to obtain the
monotonicity for the first eigenvalue of the $p$-Laplace operator,
and construct many monotonic quantities involving the first
eigenvalue of the $p$-Laplace operator along the Ricci flow under
some different curvature assumptions. By the monotonic property, we
can judge the differentiability in some sense for the first
eigenvalue of a nonlinear operator with respect to evolving metrics.
Using the same idea of our arguments, we easily see that Perelman's
eigenvalue is differentiable almost everywhere\footnote{Note that
many literatures have pointed out that the differentiability for
Perelman's eigenvalue follows from eigenvalue perturbation theory
(see also Section 2).}. From Theorem \ref{coro17} above and
Corollary \ref{T103ab} below, we also see that the first eigenvalue
of the $p$-Laplace operator is differentiable almost everywhere
along the Ricci flow on closed $2$-surfaces without any curvature
assumption. For high-dimensional case, the similar differentiability
property still holds as long as some curvature conditions are
satisfied. Of course, the proofs of these results involve many
skilled arguments and computations. Finally, it should be remarked
that it is still an open question whether its corresponding
eigenfunction is differentiable with respect to $t$-variable along
the Ricci flow.

The rest of this paper is organized as follows. In Section
\ref{Pre-Cont}, we will recall some notations about $p$-Laplace, and
prove that $\lambda_{1,p}(g(t))$ is a continuous function along the
Ricci flow. In Section \ref{sec3}, we will give Proposition
\ref{propo1}.  Using this proposition, we can finish the proof of
Theorem \ref{T101a}. In Section \ref{sec3h}, we will construct two
classes of monotonic quantities about the first eigenvalue of the
$p$-Laplace operator along the unnormalized Ricci flow. In Section
\ref{sec3b}, we will discuss the normalized Ricci flow case and
mainly prove Theorems \ref{coro17} and \ref{thm19}. In Section
\ref{comparison}, we shall prove $p$-eigenvalue comparison-type
theorem, i.e., Theorem \ref{T201}. In Section \ref{genmetr}, we will
use the same method to study the first eigenvalue of the $p$-Laplace
with respect to general evolving metrics, especially to the Yamabe
flow.

\section{Preliminaries}\label{Pre-Cont}
In this section, we will first recall some definitions about the
$p$-Laplace operator and give the definition for the first
eigenvalue of the $p$-Laplace operator under the Ricci flow on a
closed manifold. Then we will show that the first eigenvalue of the
$p$-Laplace operator is a continuous function along the Ricci flow.

Let $M^n$ be an $n$-dimensional connected closed Riemannian manifold
and $g(t)$ be a smooth solution of the Ricci flow on the time
interval $[0,T)$. Consider the nonzero first eigenvalue of the
$p$-Laplace operator $(p>1)$ at time $t$ (also called the first
$p$-eigenvalue), where $0\leq t<T$, i.e.,
\begin{equation}
\lambda_{1,p}(t):={\inf\limits_{f\neq 0}}\left\{\frac{\int_M |df|^p
d\mu}{\int_M |f|^pd\mu}:f\in W^{1,p}(M), \quad \int_M
|f|^{p-2}fd\mu=0 \right\}.
\end{equation}
Obviously, this infimum does not change when $W^{1,p}(M)$ is
replaced by $C^{\infty}(M)$. For the fixed time, this infimum is
achieved by a $C^{1,\alpha}$ ($0<\alpha<1$) eigenfunction $f_p$ (see
\cite{Serrin} and \cite{Tolk}).  The corresponding eigenfunction
$f_p$ satisfies the following Euler-Lagrange equation
\begin{equation}\label{elag}
\Delta_p f_p =-\lambda_{1,p}(t)|f_p|^{p-2}f_p,
\end{equation}
where $\Delta_p$ $(p>1)$ is the $p$-Laplace operator with respect to
$g(t)$, given by
\begin{equation}
\Delta_{p_{g(t)}}f:=\mathrm{div}_{g(t)}\left({|df|_{g(t)}^{p-2}}
df\right).
\end{equation}
If $p=2$, the $p$-Laplace operator reduces to the Laplace-Beltrami
operator. The most difference between two operators is that the
$p$-Laplace operator is a nonlinear operator in general, but the
Laplace-Beltrami operator is a linear operator.

Note that it is not clear whether the first eigenvalue of the
$p$-Laplace operator or its corresponding eigenfunction is
$C^1$-differentiable along the Ricci flow. When $p=2$, where
$\Delta_p$ is the Laplace-Beltrami operator, many papers have
pointed out that their differentiability follows from eigenvalue
perturbation theory (for example, see \cite{Cao}, \cite{Kato},
\cite{KL} and \cite{Re-Si}). But $p\neq2$, as far as we are aware,
the differentiability for the first eigenvalue of the $p$-Laplace
operator or its corresponding eigenfunction along the Ricci flow has
not been known until now. Even we have not known whether they are
locally Lipschitz. So we can not use the method used by L. Ma to
derive the monotonicity for the first eigenvalue of the $p$-Laplace
operator.

\vspace{0.5em}

Although we do not know the differentiability for
$\lambda_{1,p}(t)$, we will see that $\lambda_{1,p}(g(t))$ in fact
is a continuous function along the Ricci flow on $[0,T)$.  This is a
consequence of the following elementary result.

\begin{theorem}\label{conti}
If $g_1$ and $g_2$ are two metrics which satisfy
\[
(1+\varepsilon)^{-1}g_1\leq g_2\leq(1+\varepsilon)g_1,
\]
then for any $p>1$, we have
\begin{equation}\label{concon}
(1+\varepsilon)^{-(n+\frac
p2)}\leq\frac{\lambda_{1,p}(g_1)}{\lambda_{1,p}(g_2)}\leq
(1+\varepsilon)^{(n+\frac p2)}.
\end{equation}
In particular, $\lambda_{1,p}(g(t))$ is a continuous function in the
$t$-variable.
\end{theorem}

To prove this theorem, we first need the following fact. Let
$(M^n,g)$ be an $n$-dimensional closed Riemannian manifold. For any
non-constant function $f$, consider the following $C^1$-function on
$s\in(-\infty,\infty)$
\[
F(s):=\int_{M^n}\left|f+s\right|^pd\mu_g,\,\,\,(p>1).
\]

\begin{lemma}\label{conti2}
There exists a unique $s_0\in(-\infty,\infty)$ such that
\begin{equation}\label{claim}
F(s_0)=\min\limits_{s\in \mathbb{R}}F(s)\,\,\,\,\,\,
\mathrm{if\,\,\, and\,\,\, only\,\,\,
if}\,\,\,\,\,\,\int_M\left|f+s_0\right|^{p-2}(f+s_0)d\mu_g=0.
\end{equation}
\end{lemma}

\begin{proof}
Note that the function $|x|^p$ $(p>1)$ is a strictly convex function
on $x\in\mathbb{R}$. Meanwhile we can also check that
\[
\lim_{|s|\rightarrow{+}\infty}F(s)\rightarrow{+}\infty,\quad\quad
F'(s)=p\int_M\left|f+s\right|^{p-2}\left(f+s\right)d\mu_g.
\]
Therefore $F(s)$ is a strictly convex function and there exists a
unique $s_0\in (-\infty,+\infty)$ such that
\begin{equation}\label{claim2}
F(s_0)=\min\limits_{s\in \mathbb{R}}F(s)\,\,\,\,\,\,
\mathrm{and}\,\,\,\,\,\,
F'(s)=p\int_M\left|f+s_0\right|^{p-2}(f+s_0)d\mu_g=0.
\end{equation}
\end{proof}

Now using Lemma \ref{conti2}, we give the proof of Theorem
\ref{conti}.
\begin{proof}[Proof of Theorem \ref{conti}]
Since the volume form $d\mu$ has degree $n/2$ in $g$, we have
\begin{equation}\label{volest}
(1+\varepsilon)^{-n/2}d\mu_{g_1}\leq d\mu_{g_2}\leq
(1+\varepsilon)^{n/2}d\mu_{g_1}.
\end{equation}
Taking $f$ be the first eigenfunction of $\Delta_p$ with respect to
the metric $g_1$, we see that
\begin{equation}
\begin{aligned}\label{xianzhi1}
\lambda_{1,p}(g_1)=\frac{\int_M|df|_{g_1}^pd\mu_{g_1}}
{\int_M|f|^pd\mu_{g_1}} \,\,\,\,\,\, \mathrm{and}\,\,\,
\int_M|f|^{p-2}fd\mu_{g_1}=0.
\end{aligned}
\end{equation}
Since $\int_M|f|^{p-2}fd\mu_{g_1}=0$, Lemma \ref{conti2} implies
\[
\int_M|f|^pd\mu_{g_1}=\min\limits_{s\in
\mathbb{R}}\int_M\left|f+s\right|^pd\mu_{g_1}.
\]
Hence by (\ref{xianzhi1}), we conclude that
\begin{equation}
\begin{aligned}\label{xianzhi2}
\lambda_{1,p}(g_1)=\frac{\int_M|df|_{g_1}^pd\mu_{g_1}}
{\int_M|f|^pd\mu_{g_1}}\geq \frac{\int_M|d(f+s)|_{g_1}^pd\mu_{g_1}}
{\int_M|f+s|^pd\mu_{g_1}}.
\end{aligned}
\end{equation}

Keep in mind that under another metric $g_2$, for function
$F(s)=\int_M\left|f+s\right|^pd\mu_{g_2}$, there exists a unique
$s_0\in (-\infty,+\infty)$ such that
\begin{equation}\label{claimmn}
F(s_0)=\min\limits_{s\in \mathbb{R}}F(t)\,\,\,\,\,\,
\mathrm{and}\,\,\,\,\,\,
F'(s)=p\int_M\left|f+s_0\right|^{p-2}(f+s_0)d\mu_{g_2}=0.
\end{equation}

Using (\ref{volest}), from (\ref{xianzhi2}) we conclude that
\begin{equation}
\begin{aligned}\label{xianzhi3}
\lambda_{1,p}(g_1)\geq\frac{\int_M|d(f+s)|_{g_1}^pd\mu_{g_1}}
{\int_M|f+s|^pd\mu_{g_1}}\geq(1+\varepsilon)^{-(n+\frac
p2)}\cdot\frac{\int_M|d(f+s)|_{g_2}^pd\mu_{g_2}}
{\int_M|f+s|^pd\mu_{g_2}}.
\end{aligned}
\end{equation}
Letting $s=s_0$ in (\ref{xianzhi3}) yields
\begin{equation}
\begin{aligned}\label{xianzhi4}
\lambda_{1,p}(g_1)\geq(1+\varepsilon)^{-(n+\frac
p2)}\cdot\frac{\int_M|d(f+s_0)|_{g_2}^pd\mu_{g_2}}
{\int_M|f+s_0|^pd\mu_{g_2}}\geq(1+\varepsilon)^{-(n+\frac
p2)}\cdot\lambda_{1,p}(g_2),
\end{aligned}
\end{equation}
where for the last inequality we used
$\int_M\left|f+s_0\right|^{p-2}(f+s_0)d\mu_{g_2}=0$ and the
definition for the first $p$-eigenvalue with respect to the metric
$g_2$.

From the course of this proof, we easily see that (\ref{xianzhi4})
still holds if we exchange $g_1$ and $g_2$. Hence
\begin{equation}
(1+\varepsilon)^{-(n+\frac
p2)}\leq\frac{\lambda_{1,p}(g_1)}{\lambda_{1,p}(g_2)}\leq
(1+\varepsilon)^{(n+\frac p2)}.
\end{equation}
This completes the proof of Theorem \ref{conti}.
\end{proof}

\section{Proof of Theorem \ref{T101a}}\label{sec3}
In this section, we will prove Theorem \ref{T101a} in introduction.
In order to achieve this, we first prove the following proposition.
Our proof involves choosing a proper smooth function, which seems to
be a delicate trick.
\begin{proposition}\label{propo1}
Let $g(t)$, $t\in[0,T)$, be a solution of the unnormalized Ricci
flow (\ref{flow}) on a closed manifold $M^n$ and let
$\lambda_{1,p}(t)$ be the first eigenvalue of the $p$-Laplace
operator along this flow. For any $t_1, t_2 \in[0,T)$ and $t_2\geq
t_1$, we have
\begin{equation}\label{frded}
\lambda_{1,p}(t_2)\geq \lambda_{1,p}(t_1)+\int^{t_2}_{t_1} \mathcal
{G}(g(\xi),f(\xi))d\xi,
\end{equation}
where
\begin{equation}\label{fracww}
\mathcal{G}(g(t),f(t)):=p\int_M|df|^{p-2}Ric(\nabla f,\nabla f)
d\mu-p\int_M\Delta_pf\frac{\partial f}{\partial t}d\mu-\int_M
|df|^pRd\mu
\end{equation}
and where $f(t)$ is any $C^{\infty}$ function satisfying
$\int_M|f|^pd\mu=1$ and $\int_M|f|^{p-2}fd\mu=0$, such that at time
$t_2$, $f(t_2)$ is the corresponding eigenfunction of
$\lambda_{1,p}(t_2)$.
\end{proposition}

\begin{proof}
Set
\[
G(g(t),f(t)):=\int_M |df(t)|_{g(t)}^pd\mu_{g(t)}.
\]
We \emph{claim} that, for any time $t_2\in(0,T)$, there exists a
$C^{\infty}$ function $f(t)$ satisfying
\begin{equation}\label{eigencondis}
\int_M|f(t)|^p d\mu_{g(t)}=1 \quad\quad \mathrm{and}\quad\int_M
|f(t)|^{p{-}2}f(t)d\mu_{g(t)}=0
\end{equation}
and such that at time $t_2$, $f(t_2)$ is the eigenfunction for
$\lambda_{1,p}(t_2)$ of $\Delta_{p_{g(t_2)}}$. To see this, at time
$t_2$, we first let $f_2=f(t_2)$  be the eigenfunction for the
eigenvalue $\lambda_{1,p}(t_2)$ of $\Delta_{p_{g(t_2)}}$. Then we
consider the following smooth function
\begin{equation}\label{constr1}
h(t)=f_2\left[\frac{\mathrm{det}(g_{ij}(t_2))}{\mathrm
{det}(g_{ij}(t))}\right]^{\frac{1}{2(p-1)}}
\end{equation}
under the Ricci flow $g_{ij}(t)$. Later we normalize this smooth
function
\begin{equation}\label{constr12}
f(t)=\frac{h(t)}{\left({\int_M |h(t)|^pd\mu}\right)^{1/p}}
\end{equation}
under the Ricci flow $g_{ij}(t)$. From above, we can easily check
that $f(t)$ satisfies (\ref{eigencondis}).

By the definition for $\lambda_{1,p}(t_2)$, we have
\begin{equation}\label{equ}
\lambda_{1,p}(t_2)=G(g(t_2),f(t_2)).
\end{equation}
Notice that under the unnormalized Ricci flow,
\begin{equation}\label{eqse12}
\frac{\partial}{\partial
t}|df|^p=p|df|^{p-2}\left(R_{ij}f_if_j+f_i\frac{\partial
f_i}{\partial t}\right),\quad\quad\frac{\partial}{\partial
t}\left(d\mu\right)=-Rd\mu,
\end{equation}
where $f_i$ and $R_{ij}$ denote the covariant derivative of $f$ and
Ricci curvature with respect to the Levi-Civita connection of
$g(t)$, respectively.

Note that $G(g(t),f(t))$ is a smooth function with respect to
$t$-variable. So
\begin{equation}
\begin{aligned}\label{fded3}
\mathcal{G}(g(t),f(t)):&=\frac{d}{dt}G(g(t),f(t))\\
&=\int_M\frac{\partial}{\partial t}|df|^pd\mu-\int_M|df|^pRd\mu\\
&=p\int_M|df|^{p-2}R_{ij}f_i f_jd\mu+p\int_M|df|^{p-2}f_i
\frac{\partial}{\partial t}(f_i)d\mu-\int_M|df|^pRd\mu\\
&=p\int_M|df|^{p-2}R_{ij}f_if_jd\mu
-p\int_M\nabla_i(|df|^{p-2}f_i)\frac{\partial
f}{\partial t}d\mu-\int_M|df|^{p}Rd\mu\\
&=p\int_M|df|^{p-2}R_{ij}f_if_jd\mu-p\int_M\Delta_pf\frac{\partial
f}{\partial t}d\mu-\int_M|df|^pRd\mu,
\end{aligned}
\end{equation}
where we used (\ref{eqse12}). Taking integration on the both sides
of (\ref{fded3}) between $t_1$ and $t_2$, we conclude that
\begin{equation}\label{frac1we1}
G(g(t_2),f(t_2))-G(g(t_1),f(t_1))=\int^{t_2}_{t_1} \mathcal
{G}(g(\xi), f(\xi))d\xi,
\end{equation}
where $t_1\in[0,T)$ and $t_2\geq t_1$. Noticing
$G(g(t_1),f(t_1))\geq \lambda_{1,p}(t_1)$ and combining
({\ref{equ}}) with ({\ref{frac1we1}}), we arrive at
\[
\lambda_{1,p}(t_2)\geq \lambda_{1,p}(t_1)+\int^{t_2}_{t_1} \mathcal
{G}(g(\xi),f(\xi))d\xi,
\]
where $\mathcal {G}(g(\xi),f(\xi))$ satisfies ({\ref{fded3}}).
\end{proof}

In the following of this section, we will finish the proof of
Theorem \ref{T101a} using Proposition \ref{propo1}.

\begin{proof}[Proof of Theorem \ref{T101a}]
In fact, we only need to show that $\mathcal{G}(g(t),f(t))>0$ in
Proposition \ref{propo1}. Notice that at time $t_2$,
$\lambda_{1,p}(t_2)$ is the first eigenvalue and $f(t_2)$ is the
corresponding eigenfunction. Therefore at time $t_2$, we have
\begin{equation}
\begin{aligned}\label{guoch}
\mathcal{G}(g(t_2),f(t_2))&=p\int_M|df|^{p-2}R_{ij}f_if_jd\mu
-p\int_M\Delta_pf\frac{\partial f}{\partial t}d\mu-\int_M|df|^pRd\mu\\
&=p\int_M|df|^{p{-}2}R_{ij}f_if_jd\mu+p\lambda_{1,p}(t_2)
\int_M|f|^{p-2}f\frac{\partial f}{\partial t}d\mu\int_M|df|^pRd\mu,
\end{aligned}
\end{equation}
where we used
$\Delta_pf(t_2)=-\lambda_{1,p}(t_2)|f(t_2)|^{p-2}f(t_2)$.

Under the unnormalized Ricci flow, from the constraint condition
\[
\frac{d}{dt}\int_M \left|f(t)\right|^p d\mu_{g(t)}=0,
\]
we know that
\begin{equation}\label{xianz}
p\int_M |f|^{p-2}f \frac{\partial f}{\partial t}
d\mu=\int_M|f|^pRd\mu.
\end{equation}
Substituting this into the above formula (\ref{guoch}) and combining
the assumption of Theorem \ref{T101a}:
$R_{ij}-\tfrac{R}{p}g_{ij}\geq -\epsilon g_{ij}$ in
$M^n\times[0,T)$, we obtain
\begin{equation}
\begin{aligned}\label{guoch2}
\mathcal{G}(g(t_2),f(t_2))&=p\int_M|df|^{p-2}R_{ij}f_if_jd\mu
+p\lambda_{1,p}(t_2)\int_M|f|^{p-2}f
\frac{\partial f}{\partial t}d\mu-\int_M|df|^pRd\mu\\
&=\lambda_{1,p}(t_2)\int_M|f|^p R d\mu
+\int_M|df|^{p-2}(pR_{ij}-Rg_{ij})f_if_jd\mu\\
&\geq\lambda_{1,p}(t_2)\int_M|f|^p R d\mu
-p\cdot\epsilon\int_M|df|^pd\mu\\
&=\lambda_{1,p}(t_2)\int_M|f|^p(R-p\cdot\epsilon)d\mu.
\end{aligned}
\end{equation}
Meanwhile we also have another assumption of Theorem \ref{T101a} on
the scalar curvature
\[
R\geq p\cdot\epsilon\,\,\,\mathrm{and}\,\,\, R\not\equiv
p\cdot\epsilon\quad \quad\mathrm{in}\quad M^n\times\{0\}.
\]
It is well-known that $R\geq p\cdot\epsilon$ is preserved by the
unnormalized Ricci flow. Furthermore by the strong maximum principle
(for example, see Proposition 12.47 of Chapter 12 in \cite{CCG2}),
we conclude that
\begin{equation}\label{remain}
R>p\cdot\epsilon\quad \quad\mathrm{in}\quad M^n\times[0,T).
\end{equation}
Plugging this into (\ref{guoch2}) implies
$\mathcal{G}(g(t_2),f(t_2))>0$. Notice that $f(x,t)$ is a smooth
function with respect to $t$-variable. Therefore we can arrive at
$\mathcal{G}(g(\xi),f(\xi))>0 $ in any sufficient small neighborhood
of $t_2$. Hence
\begin{equation}
\int^{t_2}_{t_1}\mathcal{G}(g(\xi),f(\xi))d\xi>0
\end{equation}
for any $t_1<t_2$ sufficiently close to $t_2$. In the end, by
Proposition \ref{propo1}, we conclude
\[
\lambda_{1,p}(t_2)>\lambda_{1,p}(t_1)
\]
for any $t_1<t_2$ sufficiently close to $t_2$. Since $t_2\in[0,T)$
is arbitrary, then the first part of Theorem \ref{T101a} follows.

As for the differentiability for $\lambda_{1,p}(t)$, since
$\lambda_{1,p}(t)$ is increasing on the time interval $[0,T)$ under
curvature conditions of the theorem, by the classical Lebesgue's
theorem (for example, see Chapter 4 in \cite{Mu-Ko}), it is easy to
see that $\lambda_{1,p}(t)$ is differentiable almost everywhere on
$[0,T)$.
\end{proof}

\begin{remark}
(1). Our proof of the first $p$-eigenvalue monotonicity is not
derived from the differentiability for $\lambda_{1,p}(t)$ or its
corresponding eigenfunction. In fact we do not know whether they are
differentiable in advance. It would be interesting to find out
whether the corresponding eigenfunction of the $p$-Laplace operator
is a $C^1$-differentiable function with respect to $t$-variable
along the Ricci flow on a closed manifold $M^n$. If it is true, we
can use L. Ma's method to get our result.

(2). If $p=2$, the above theorem is similar to L. Ma's main result
for the first eigenvalue of the Laplace operator in \cite{Ma}.

(3). Using this method, we can not get any monotonicity for higher
order eigenvalues of the $p$-Laplace operator.
\end{remark}

\section{Monotonic quantities along unnormalized Ricci flow}\label{sec3h}
Motivated by the works of X.-D. Cao \cite{Cao1} and \cite{Cao}, in
this section, we first introduce a new smooth eigenvalue function
(see (\ref{geneig}) below), and then we give the following useful
Lemma \ref{genlemma}, resembling Proposition \ref{propo1} of Section
\ref{sec3}. Using this lemma, we can obtain two classes of
interesting monotonic quantities along the unnormalized Ricci flow,
that is, Theorem \ref{T10b}, Theorem \ref{norma1} and Corollary
\ref{T103aa}. Then by means of those monotonic quantities, we can
prove the differentiability for the first eigenvalue of the
$p$-Laplace operator along the unnormalized Ricci flow.

Let $M^n$ be an $n$-dimensional connected closed Riemannian manifold
and $\tilde{g}(\tilde{t})$ be a smooth solution of the normalized
Ricci flow on the time interval $[0,\infty)$. Now we can define a
general smooth eigenvalue function
\begin{equation}\label{geneig}
\lambda_{1,p}(\tilde{f},\tilde{t}):=
\int_M\tilde{\Delta}_{p_{\tilde{g}(\tilde{t})}}\tilde{f}\cdot
\tilde{f}d\tilde{\mu}=\int_M |d\tilde{f}|^pd\tilde{\mu},
\end{equation}
where $\tilde{f}$ is a smooth function and satisfies the following
equalities
\begin{equation}\label{conRn}
\int_M |\tilde{f}(\tilde{t})|^p
d\tilde{\mu}_{\tilde{g}(\tilde{t})}=1\quad\quad\mathrm{and}\quad\int_M
|\tilde{f}(\tilde{t})|^{p-2}\tilde{f}(\tilde{t})
d\tilde{\mu}_{\tilde{g}(\tilde{t})}=0.
\end{equation}
From the proof of Proposition \ref{propo1},  we see that the above
restriction (\ref{conRn}) can be achieved.

Obviously, at time $t_0$, if $\tilde{f}$ is the corresponding
eigenfunction of the first eigenvalue $\lambda_{1,p}(t_0)$, then
\[
\lambda_{1,p}(\tilde{f},t_0)=\lambda_{1,p}(t_0).
\]

For the convenient of writing, we shall drop the tilde over all the
variables used above to distinguish between the normalized and
unnormalized Ricci flow.

\begin{lemma}\label{genlemma}
If $\lambda_{1,p}(t)$ is the first eigenvalue of
$\Delta_{p_{g(t)}}$, whose metric satisfying the normalized Ricci
flow and $f(t_0)$ is the corresponding eigenfunction of
$\lambda_{1,p}(t)$ at time $t_0$, then we have
\begin{equation}
\begin{aligned}\label{dd5}
\frac{d}{dt}\lambda_{1,p}(f,t)\Big|_{t=t_0}
=&\lambda_{1,p}(f(t_0),t_0)\int_M|f|^p Rd\mu
+p\int_M|df|^{p-2}R_{ij}f_i f_j d\mu\\
&-\int_M|df|^pR d\mu-\frac pn r\lambda_{1,p}(f(t_0),t_0).
\end{aligned}
\end{equation}
In particular, for any closed $2$-surface, we have
\begin{equation}
\begin{aligned}\label{kkdd5}
\frac{d}{dt}\lambda_{1,p}(f,t)\Big|_{t=t_0}&=\lambda_{1,p}(f(t_0),t_0)\int_M|
f|^pRd\mu+\left(\frac p2-1\right)\int_M|df|^pR d\mu\\
&\,\,\,\,\,\,-\frac p2 r\lambda_{1,p}(f(t_0),t_0),
\end{aligned}
\end{equation}
where $f$ evolves by (\ref{conRn}) with the initial data $f(t_0)$.
\end{lemma}

\begin{proof}
The proof is by direct computations. Here we need to use
\[
\frac{\partial}{\partial t}|df|^p=p| df|^{p-2}
\left(R_{ij}f_if_j-\frac rn g_{ij}f_i f_j+f_i\frac{\partial
f_i}{\partial t}\right),\,\,\,\,\,\,\frac{\partial}{\partial
t}(d\mu)=(r-R)d\mu.
\]
Then
\begin{equation}
\begin{aligned}\label{frack6}
\frac{d\lambda_{1,p}(f,t)}{dt}\Big|_{t=t_0}
&=p\int_M|df|^{p-2}R_{ij}f_if_jd\mu+p\int_M|df|^{p-2}f_i
\frac{\partial\left(f_i\right)}{\partial t}d\mu\\
&\,\,\,\,\,\,-p\int_M|df|^p\frac rnd\mu+\int_M|df|^p(r-R)d\mu\\
&=p\int_M|df|^{p-2}R_{ij}f_if_j d\mu-p\int_M \nabla_i\left(|
df|^{p-2}f_i\right)\frac{\partial f}{\partial t}d\mu\\
&\,\,\,\,\,\,-\frac pn r\lambda_{1,p}(f(t_0),t_0)+\int_M|df|^p(r-R)d\mu\\
&=p\int_M|df|^{p-2}R_{ij}f_if_j d\mu
+p\lambda_{1,p}(f(t_0),t_0)\int_M|f
|^{p-2}f\frac{\partial f}{\partial t}d\mu\\
&\,\,\,\,\,\,-\frac pn r\lambda_{1,p}(f(t_0),t_0)
+\int_M|df|^p(r-R)d\mu,
\end{aligned}
\end{equation}
where we used $f$ is the eigenfunction at time $t_0$, i.e., equation
(\ref{elag}) at time $t_0$. Note that by (\ref{conRn}), we have
\begin{equation}\label{frac7}
p\int_M |f|^{p-2}f \frac{\partial f}{\partial t}d\mu=\int_M
|f|^p(R-r)d\mu.
\end{equation}
Plugging this into ({\ref{frack6}}) yields the desired (\ref{dd5}).
For any closed $2$-surface, we have $R_{ij}=\frac R2 g_{ij}$. Hence
(\ref{kkdd5}) follows from (\ref{dd5}).
\end{proof}

\begin{remark}
In \cite{Wu2}, the first author used a similar method and proved a
similar result for the unnormalized Ricci flow (see Proposition 2.1
in \cite{Wu2}).
\end{remark}

\vspace{0.5em}

In the following we first obtain increasing quantities along the
unnormalized Ricci flow by using Lemma \ref{genlemma}.
\begin{theorem}\label{T10b}
Let $g(t)$ and $\lambda_{1,p}(t)$ $(p>1)$ be the same as in Theorem
\ref{T101a}. If $\rho_0:=\inf_{M}R(0)>0$ and
\begin{equation}\label{Normatiao1}
R_{ij}-\tfrac{R}{p}g_{ij}(t)>0\quad \quad \mathrm{in}\quad
M^n\times[0,T),
\end{equation}
then the following quantity
\begin{equation}\label{Normatiao2}
\lambda_{1,p}(t)\cdot\left(\rho_0^{-1}-2at\right)^{\frac{1}{2a}},
\end{equation}
is strictly increasing and therefore $\lambda_{1,p}(t)$ is
differentiable almost everywhere along the unnormalized Ricci flow
on $[0,T')$, where $a{:=}\max\{\frac1n, \frac{n}{p^2}\}$ and
$T'{:=}\min\{\frac{1}{2a\rho_0},T\}$.
\end{theorem}
\begin{proof}
We assume that at time $t_0\in[0,T)$, if $g$ is the corresponding
eigenfunction of $\lambda_{1,p}(t_0)$, then under the unnormalized
Ricci flow, we can construct a smooth function $f$ satisfying
\[
\int_M|f(t)|^p d\mu_{g(t)}=1 \quad\quad \mathrm{and}
\quad\int_M|f(t)|^{p{-}2}f(t)d\mu_{g(t)}=0,
\]
and such that at time $t=t_0$, $f=g$ is the eigenfunction of
$\lambda_{1,p}(t_0)$. Meanwhile we can define a general smooth
eigenvalue function $\lambda_{1,p}(f, t)$ as (\ref{geneig}) under
the unnormalized Ricci flow. Obviously, we have
\[
\lambda_{1,p}(f(t_0),t_0)=\lambda_{1,p}(t_0).
\]
According to (\ref{dd5}) of Lemma \ref{genlemma}, we have
\begin{equation}\label{dd5k}
\frac{d}{dt}\lambda_{1,p}(f,t)\Big|_{t=t_0}
=\lambda_{1,p}(f(t_0),t_0)\int_M|f|^pRd\mu
+\int_M|df|^{p-2}(pR_{ij}-Rg_{ij})f_if_j d\mu,
\end{equation}
where $f$ is a smooth function satisfying the above assumptions. By
the assumption $R_{ij}-\frac{R}{p}g_{ij}>0$ of Theorem \ref{T10b},
we get
\begin{equation}\label{dd5guan}
\frac{d}{dt}\lambda_{1,p}(f,t)\Big|_{t=t_0}
>\lambda_{1,p}(f(t_0),t_0)\int_M|f|^pRd\mu.
\end{equation}
The evolution of the scalar curvature $R$ under the unnormalized
Ricci flow
\[
\frac{\partial}{\partial t}R=\Delta R+2|Ric|^2
\]
and inequality $|Ric|^2\geq a{R^2}$ ($a:=\max \{\frac1n,
\frac{n}{p^2}\}$) imply
\begin{equation}
\begin{aligned}\label{frac10}
\frac{\partial}{\partial t}R \geq \Delta R+2aR^2.
\end{aligned}
\end{equation}
Since the solutions to the corresponding ODE
\[
{d\rho}/{dt}=2a\rho^2
\]
are
\begin{equation*}
\begin{aligned}
\rho(t)=\frac{1}{{\rho_0}^{-1}-2at},\quad t\in[0,T'),
\end{aligned}
\end{equation*}
where $\rho_0:=\inf_{M} R(0)$ and $T':=\min\{(2a\rho_0)^{-1},T\}$.
Using the maximum principle to ({\ref{frac10}}), we have $R(x,t)\geq
\rho(t)$. Therefore (\ref{dd5guan}) becomes
\[
\frac{d}{dt}\lambda_{1,p}(f,t)\Big|_{t=t_0}
>\lambda_{1,p}(f(t_0),t_0)\cdot\rho(t_0).
\]
Note that $\lambda_{1,p}(f,t)$ and $\rho(t)$ are both smooth
functions with respect to $t$-variable. Hence we have
\begin{equation}\label{dianzhi2}
\frac{d}{dt}\lambda_{1,p}(f,t)>\lambda_{1,p}(f(t),t)\cdot\rho(t)
\end{equation}
in any sufficiently small neighborhood of $t_0$. Now integrating the
above inequality with respect to time $t$ on time interval
$[t_1,t_0]$, we get
\begin{equation}
\begin{aligned}\label{integ1q}
\ln&\lambda_{1,p}(f(t_0),t_0)-\ln\lambda_{1,p}(f(t_1),t_1)\\
&>\left(-\frac{1}{2a}\right)\cdot\ln\left(\rho_0^{-1}-2at\right)
\Big|_{t=t_0}-\left(-\frac{1}{2a}\right)
\cdot\ln\left(\rho_0^{-1}-2at\right)\Big|_{t=t_1}
\end{aligned}
\end{equation}
for any $t_1<t_0$ sufficiently close to $t_0$. Note that
$\lambda_{1,p}(f(t_0),t_0)=\lambda_{1,p}(t_0)$ and
$\lambda_{1,p}(f(t_1),t_1)\geq\lambda_{1,p}(t_1)$. Then
(\ref{integ1q}) becomes
\[
\ln\lambda_{1,p}(t_0)+\ln\left(\rho_0^{-1}-2at_0\right)^{\frac{1}{2a}}
>\ln\lambda_{1,p}(t_1)+\ln\left(\rho_0^{-1}-2at_1\right)^{\frac{1}{2a}}.
\]
Namely,
\[
\lambda_{1,p}(t_0)\cdot\left(\rho_0^{-1}-2at_0\right)^{\frac{1}{2a}}
>\lambda_{1,p}(t_1)\cdot\left(\rho_0^{-1}-2at_1\right)^{\frac{1}{2a}}
\]
for any $t_1<t_0$ sufficiently close to $t_0$. Since $t_0$ is
arbitrary, then (\ref{Normatiao2}) follows.

Now we know that
\[
\lambda_{1,p}(t)\cdot\left(\rho_0^{-1}-2at\right)^{\frac{1}{2a}}
\]
is increasing along the unnormalized Ricci flow. Moreover,
$\left(\rho_0^{-1}-2at\right)^{\frac{1}{2a}}$ is a smooth function.
Hence by the Lebesgue's theorem, $\lambda_{1,p}(t)$ is
differentiable almost everywhere along the unnormalized Ricci flow
on $[0,T')$.
\end{proof}

\begin{remark}
Since function $\left(\rho_0^{-1}-2at\right)^{\frac{1}{2a}}$ is
decreasing in $t$-variable, Theorem \ref{T10b} also implies that
$\lambda_{1,p}(t)$ is strictly increasing along the unnormalized
Ricci flow on $[0,T')$.
\end{remark}

\vspace{0.5em}

We also have decreasing quantities along the unnormalized Ricci
flow.
\begin{theorem}\label{norma1}
Let $g(t)$ and $\lambda_{1,p}(t)$ $(p>1)$ be the same as in Theorem
\ref{T101a}. If
\begin{equation}\label{assumpa}
0\leq R_{ij}<\tfrac{R}{p}g_{ij}(t)\quad \quad\mathrm{in}\quad
M^n\times[0,T),
\end{equation}
then the following quantity
\begin{equation}\label{Normatiao2k}
\lambda_{1,p}(t)\cdot\left(\sigma_0^{-1}
-\frac{2n}{p^2}t\right)^{\frac{p^2}{2n}}
\end{equation}
is strictly decreasing and therefore $\lambda_{1,p}(t)$ is
differentiable almost everywhere along the unnormalized Ricci flow
on $[0,T')$, where $\sigma_0:=\sup_MR(0)$ and
$T':=\min\{\frac{p^2}{2n\sigma_0},T\}$.
\end{theorem}
\begin{proof}
The proof is similar to that of Theorem \ref{T10b} with the
difference that we need to estimate the upper bounds of the right
hand side of (\ref{djklmm}). Here we only briefly sketch the proof.
According to (\ref{dd5}) of Lemma \ref{genlemma}, we have
\begin{equation}\label{djklmm}
\frac{d}{dt}\lambda_{1,p}(f,t)\Big|_{t=t_0}
=\lambda_{1,p}(f(t_0),t_0)\int_M|f|^pRd\mu
+\int_M|df|^{p-2}(pR_{ij}-Rg_{ij})f_i f_j d\mu,
\end{equation}
where $f$ is a smooth function satisfying the same assumptions as in
the proof of Theorem \ref{T10b}.

Note that $0\leq R_{ij}<\frac Rp g_{ij}$ implies $|Ric|^2<
\frac{n}{p^2}R^2$. So the evolution of the scalar curvature $R$
under the unnormalized Ricci flow
\[
\frac{\partial}{\partial t}R=\Delta R+2|Ric|^2
\]
implies
\begin{equation}
\begin{aligned}\label{frac100}
\frac{\partial}{\partial t}R \leq \Delta R+\frac{2n}{p^2}R^2.
\end{aligned}
\end{equation}
Applying the maximum principle to (\ref{frac100}), we have
\[
0\leq R(x,t)\leq \sigma(t),
\]
where
\begin{equation*}
\sigma(t)=\frac{1}{{\sigma_0}^{-1}-\tfrac{2n}{p^2}t},\quad
t\in[0,T'),
\end{equation*}
and where $\sigma_0:=\sup_M R(0)$ and
$T':=\min\{\frac{p^2}{2n\sigma_0},T\}$.

Substituting $0\leq R(x,t)\leq \sigma(t)$ and $0\leq R_{ij}<\frac Rp
g_{ij}$ into (\ref{djklmm}) yields
\[
\frac{d}{dt}\lambda_{1,p}(f,t)\Big|_{t=t_0}<\lambda_{1,p}(f(t_0),t_0)
\cdot\sigma(t_0).
\]
Hence
\[
\frac{d}{dt}\lambda_{1,p}(f,t)<\lambda_{1,p}(f(t),t)\cdot\sigma(t)
\]
in any sufficiently small neighborhood of $t_0$. Integrating this
inequality with respect to time $t$ on time interval $[t_0,t_1]$
yields
\[
\lambda_{1,p}(t_1)\cdot\left(\sigma_0^{-1}
-\frac{2n}{p^2}t_1\right)^{\frac{p^2}{2n}}
<\lambda_{1,p}(t_0)\cdot\left(\sigma_0^{-1}
-\frac{2n}{p^2}t_0\right)^{\frac{p^2}{2n}}
\]
for any $t_1>t_0$ sufficiently close to $t_0$, where we used
$\lambda_{1,p}(f,t_0)=\lambda_{1,p}(t_0)$ and
$\lambda_{1,p}(f,t_1)\geq\lambda_{1,p}(t_1)$. Since $t_0$ is
arbitrary, then Theorem \ref{norma1} follows.
\end{proof}

\vspace{0.5em}

For any closed $3$-manifold, we have
\begin{corollary}\label{T103aa}
Let $g(t)$ and $\lambda_{1,p}(t)$ be the same as in Theorem
\ref{T101a}., where we assume $n=3$ and $1<p<3$. If
\begin{equation}\label{initialc1}
0\leq R_{ij}(0)<\tfrac{R(0)}{p}g_{ij}(0)\quad\quad\mathrm{in}\quad
M^3\times\{0\},
\end{equation}
then the conclusion of Theorem \ref{norma1} is also true.
\end{corollary}

\begin{remark}
Note that if $p=2$, condition (\ref{initialc1}) is the same as
positive sectional curvatures of this closed manifold.
\end{remark}

\begin{proof}
According to Hamilton's maximum principle for tensors (see Theorem
9.6 in \cite{Hamilton}), for $1<p<3$, we conclude that $0\leq
R_{ij}<\tfrac{R}{p}g_{ij}$ is preserved under the Ricci flow.
Therefore the desired conclusion follows from Theorem \ref{norma1}.
\end{proof}

\section{First $p$-eigenvalue along normalized Ricci flow}\label{sec3b}
In this section, we will first discuss the differentiability for
$\lambda_{1,p}(\tilde{g}(\tilde{t}))$ under normalized Ricci flow by
means of the differentiability for $\lambda_{1,p}(g(t))$ under
unnormalized Ricci flow. Then for closed $2$-surfaces, we obtain
many monotonic quantities about the first eigenvalue of the
$p$-Laplace operator along the normalized Ricci flow  without any
curvature assumption, that is, Theorems \ref{coro17} and \ref{thm19}
in introduction.

At first we can apply the differentiability for
$\lambda_{1,p}(g(t))$ under the unnormalized Ricci flow to derive
the differentiability for $\lambda_{1,p}(\tilde{g}(\tilde{t}))$
under the normalized case.
\begin{theorem}\label{T100b}
Let $\tilde{g}(\tilde{t})$, $\tilde{t}\in[0,\infty)$, be a solution
of the normalized Ricci flow (\ref{norflow}) on a closed manifold
$M^n$ and let $\lambda_{1,p}(\tilde{t})$ be the first eigenvalue of
the $p$-Laplace operator of the metric $\tilde{g}(\tilde{t})$. If
the curvature assumptions of Theorem \ref{T101a} (Theorem
\ref{T10b}, Theorem \ref{norma1} or Corollary \ref{T103aa}) are
satisfied, then $\lambda_{1,p}(\tilde{t})$ is differentiable almost
everywhere along the normalized Ricci flow on $[0,\infty)$ in each
case.
\end{theorem}

\begin{proof}[Proof of Theorem \ref{T100b}]
Under the normalized Ricci flow $\tilde{g}(\tilde{t}):=c(t)g(t)$, we
have
\begin{equation}
\begin{aligned}\label{connor}
\lambda_{1,p}(\tilde{g}(\tilde{t}))=
\frac{\int_M|d\tilde{f}|^p_{\tilde{g}(\tilde{t})}d\tilde{\mu}}
{\int_M|\tilde{f}|^pd\tilde{\mu}}
=\frac{\int_M|d\tilde{f}|^p_{\tilde{g}(\tilde{t})}d\mu}
{\int_M|\tilde{f}|^pd\mu}
=c(t)^{-p/2}\frac{\int_M|d\tilde{f}|^p_{g(t)}d\mu}
{\int_M|\tilde{f}|^pd\mu},
\end{aligned}
\end{equation}
where $\tilde{f}$ is the eigenfunction for the first eigenvalue
$\lambda_{1,p}(\tilde{t})$ with respect to $\tilde{g}(\tilde{t})$,
which implies $\int_M |\tilde{f}|^{p-2}\tilde{f}d\tilde{\mu}=0$.
Since $\tilde{g}(\tilde{t}):=c(t)g(t)$, we also have
\[
\int_M |\tilde{f}|^{p-2}\tilde{f}d\mu=0.
\]
Consider the following quantity
\begin{equation}\label{quant1}
\frac{\int_M|d\phi|^p_{g(t)}d\mu} {\int_M|\phi|^pd\mu},
\end{equation}
where $\phi$ is any $C^1$ function. Clearly, if $\phi=\tilde{f}$,
then (\ref{quant1}) achieves its minimum. If it is not true, this
contradicts (\ref{connor}) by choosing $c(t)=1$. Therefore
(\ref{connor}) implies that
\[
\lambda_{1,p}(\tilde{g}(\tilde{t}))=c(t)^{-p/2}\cdot\lambda_{1,p}(g(t)).
\]
Note that $\lambda_{1,p}(g(t))$ is differentiable almost everywhere
under the curvature assumptions of Theorem \ref{T101a} (Theorem
\ref{T10b}, Theorem \ref{norma1} or Corollary \ref{T103aa}) and
$c(t)$ is a smooth function. Hence $\lambda_{1,p}(\tilde{t})$ is
differentiable almost everywhere in each case along the normalized
Ricci flow on $[0,\infty)$.
\end{proof}

\begin{remark}
For any $2$-surface, we claim that $\lambda_{1,p}(t)$ is
differentiable almost everywhere along the Ricci flow without any
curvature assumption (see Theorems \ref{coro17} and \ref{thm19}, and
Corollary \ref{T103ab}).
\end{remark}

In the rest of this section, we shall discuss the monotonic
quantities about the first eigenvalue of the $p$-Laplace operator
along the normalized Ricci flow on closed $2$-surfaces. From this,
we also see that $\lambda_{1,p}(t)$ is differentiable almost
everywhere along the normalized Ricci flow without any curvature
assumption.

We recall the following curvature estimates along the normalized
Ricci flow on closed surfaces (see Proposition 5.18 in
\cite{Chow-Knopf}).

\begin{proposition}\label{procurest}
For any solution $(M^2,g(t))$ of the normalized Ricci flow on a
closed surface, there exists a constant $C>0$ depending only on the
initial metric such that:
\begin{enumerate}
\item If $r<0$, then $r-Ce^{rt}\leq R\leq r+Ce^{rt}$.

\item If $r=0$, then $-\frac{C}{1+Ct}\leq R\leq C$.

\item If $r>0$, then $-Ce^{rt}\leq R\leq r+Ce^{rt}$.
\end{enumerate}
\end{proposition}

Now using Proposition \ref{procurest}, we shall prove Theorem
\ref{coro17}. The method of proof is almost the same as that of
Theorem \ref{T10b}.
\begin{proof}[Proof of Theorem \ref{coro17}]

\emph{Step 1}: we first prove the case $p\geq 2$. Since $n=2$, by
(\ref{kkdd5}) of Lemma \ref{genlemma}, under the normalized Ricci
flow, we have
\begin{equation}
\begin{aligned}\label{kkddn2}
\frac{d}{dt}\lambda_{1,p}(f,t)\Big|_{t=t_0}
&=\lambda_{1,p}(f(t_0),t_0)\int_M|f|^pRd\mu
+\left(\frac p2-1\right)\int_M|df|^pR d\mu\\
&\,\,\,\,\,\,-\frac p2 r\lambda_{1,p}(f(t_0),t_0),
\end{aligned}
\end{equation}
where $f$ is defined by Lemma \ref{genlemma}.

Case 1: $\chi(M^2)<0$.

Note that the evolution of the scalar curvature $R$ on a closed
surface under the normalized Ricci flow is
\begin{equation}\label{evolutionsur}
\frac{\partial}{\partial t}R = \Delta R+R(R-r).
\end{equation}
By the Gauss-Bonnet theorem, $r$ is determined by the Euler
characteristic $\chi(M^2)$, i.e., $r=4\pi
\chi(M^2)/\mathrm{Area}{(M^2)}$. Now if $\chi(M^2)<0$, applying the
maximum principle to equation (\ref{evolutionsur}), we obtain sharp
lower bounds of
 the scalar curvature $R$:
\begin{equation}\label{RRRk}
R(x,t)\geq\displaystyle \frac{r}{1-(1-\tfrac{r}{\rho_0}) e^{rt}},
\quad \quad t\in[0,\infty).
\end{equation}
Note that in this setting, we need more accurate lower bounds than
Proposition \ref{procurest}. By inequality (\ref{RRRk}), we have
\begin{equation}\label{RRR2k}
R(x,t)>\frac{r}{1-(1-\tfrac{r}{\rho_0}) e^{rt}}-\epsilon, \quad
\quad t\in[0,\infty)
\end{equation}
for $\epsilon>0$ sufficiently small. Substituting this into the
above formula (\ref{kkddn2}), we obtain
\begin{equation}
\begin{aligned}\label{dianzhik}
\frac{d\lambda_{1,p}(f,t)}{dt}\Big|_{t=t_0}
&>\lambda_{1,p}(f(t_0),t_0)\left[\frac{r}{1-(1-\frac {r}{\rho_0}
)e^{rt_0}}-\frac p2 r\right]\\
&\,\,\,\,\,\,+\left(\frac p2-1\right)
\frac{r\lambda_{1,p}(f(t_0),t_0)}{1-(1-\frac{r}{\rho_0})e^{rt_0}}
-\frac{p\epsilon}{2}\lambda_{1,p}(f(t_0),t_0)\\
\,\,\,\,\,\,&=\frac p2\lambda_{1,p}(f(t_0),t_0)\left[
\frac{r}{1-(1-\frac{r}{\rho_0} )e^{rt_0}}-r-\epsilon\right].
\end{aligned}
\end{equation}
Since $\lambda_{1,p}(f,t)$ is a smooth function with respect to
$t$-variable, we have
\begin{equation}\label{dianzhi2k}
\frac{d}{dt}\lambda_{1,p}(f,t)>\frac p2\lambda_{1,p}(f(t),t)
\left[\frac{r}{1-(1-\frac {r}{\rho_0} )e^{rt}}-r-\epsilon\right]
\end{equation}
in any sufficiently small neighborhood of $t_0$. Integrating the
above inequality with respect to time $t$ on a sufficiently small
time interval $[t_1,t_0]$, we obtain
\begin{equation}
\begin{aligned}\label{integ1}
\ln&\lambda_{1,p}(f(t_0),t_0)-\ln\lambda_{1,p}(f(t_1),t_1)\\
&>\frac p2\left[\ln\frac{\frac{r}{\rho_0}
e^{rt_0}}{1-(1-\frac{r}{\rho_0})e^{rt_0}}-(r+\epsilon)t_0\right]
-\frac p2 \left[\ln\frac{\frac{r}{\rho_0}
e^{rt_1}}{1-(1-\frac{r}{\rho_0})e^{rt_1}}-(r+\epsilon)t_1\right]
\end{aligned}
\end{equation}
for any $t_1<t_0$ sufficiently close to $t_0$ (Note that $t_1$ may
equal to $0$). Since $\lambda_{1,p}(f(t_0),t_0)=\lambda_{1,p}(t_0)$
and $\lambda_{1,p}(f(t_1),t_1)\geq\lambda_{1,p}(t_1)$, then we have
\begin{equation}
\begin{aligned}\label{integ2}
\ln&\lambda_{1,p}(t_0)-\ln\lambda_{1,p}(t_1)\\
&>\frac p2\left[\ln\frac{\frac{r}{\rho_0}
e^{rt_0}}{1-(1-\frac{r}{\rho_0})e^{rt_0}}-(r+\epsilon)t_0\right]
-\frac p2 \left[\ln\frac{\frac{r}{\rho_0}
e^{rt_1}}{1-(1-\frac{r}{\rho_0})e^{rt_1}}-(r+\epsilon)t_1\right]
\end{aligned}
\end{equation}
for any $t_1<t_0$ sufficiently close to $t_0$. Since $t_0$ is
arbitrary, we conclude that
\begin{equation}\label{integ3}
\ln\lambda_{1,p}(t)-\frac p2\left[\ln\frac{\frac{r}{\rho_0}
e^{rt}}{1-(1-\frac{r}{\rho_0})e^{rt}}-(r+\epsilon)t\right]
\end{equation}
is increasing along the normalized Ricci flow. Taking
$\epsilon\rightarrow 0$, we know that
\begin{equation}\label{integtmm}
\ln\left[\lambda_{1,p}(t)\cdot\left(\frac{\rho_0}{r}
-\frac{\rho_0}{r}e^{rt}+e^{rt}\right)^{p/2}\right]
\end{equation}
is non-decreasing along the normalized Ricci flow. By the Lebesgue's
theorem, (\ref{integtmm}) is differentiable almost everywhere along
the normalized Ricci flow on $[0,\infty)$. We also note that
\[
\left[\frac{\rho_0}{r}-\frac{\rho_0}{r}e^{rt}+e^{rt}\right]^{p/2}
\]
is a smooth function. Hence $\lambda_{1,p}(t)$ is differentiable
almost everywhere along the normalized Ricci flow.

Case 2: $\chi(M^2)=0$.

If $\chi(M^2)=0$, i.e., $r=0$, by Proposition \ref{procurest}, we
have
\begin{equation}\label{RRRkevo2}
R(x,t)\geq-\frac{C}{1+Ct}.
\end{equation}
Substituting this into formula (\ref{kkddn2}) and applying similar
arguments above (in case of $\chi(M^2)\neq0$), we can obtain the
desired results.

Case 3: $\chi(M^2)>0$.

This proof is similar to the proof of Case 2. we still use
Proposition \ref{procurest} and formula (\ref{kkddn2}).

\vspace{0.5em}

\emph{Step 2}: we consider the case $1<p<2$. Since the method of
proof is similar to the previous discussions, we only give some key
computations.

Case 1: $\chi(M^2)<0$.

By (\ref{kkddn2}) and $R\leq r+Ce^{rt}$ of Proposition
\ref{procurest}, we have
\begin{equation}
\begin{aligned}\label{kkddn2oy}
\frac{d}{dt}\lambda_{1,p}(f,t)\Big|_{t=t_0}
&\geq\lambda_{1,p}(f(t_0),t_0)
\left[\frac{r}{1-(1-\frac{r}{\rho_0})e^{rt_0}}+\left(\frac
p2-1\right)\left(r+Ce^{rt_0}\right)-\frac p2 r\right]\\
&=\lambda_{1,p}(f(t_0),t_0)
\left[\frac{r}{1-(1-\frac{r}{\rho_0})e^{rt_0}}-r+\left(\frac
p2-1\right)Ce^{rt_0}\right]
\end{aligned}
\end{equation}
where $f$ is defined by Lemma \ref{genlemma}.

Following similar arguments above, we conclude that (\ref{kkddn2oy})
still holds in any sufficiently small neighborhood of $t_0$. Then
integrating this inequality with respect to time $t$ on a
sufficiently small time interval $[t_1,t_0]$, we obtain
\begin{equation}
\begin{aligned}\label{integp1}
\ln\lambda_{1,p}(f(t_0),t_0){-}\ln\lambda_{1,p}(f(t_1),t_1)
&\geq\left[\ln\frac{\frac{r}{\rho_0}}{1-(1-\frac{r}{\rho_0})e^{rt_0}}
{+}\left(\frac p2-1\right)\frac{C}{r}e^{rt_0}\right]\\
&\,\,\,\,\,\,-\left[\ln\frac{\frac{r}{\rho_0}}{1-(1-\frac{r}{\rho_0})e^{rt_1}}
{+}\left(\frac p2-1\right)\frac{C}{r}e^{rt_1}\right]
\end{aligned}
\end{equation}
for any $t_1<t_0$ sufficiently close to $t_0$. Note that
$\lambda_{1,p}(f(t_0),t_0)=\lambda_{1,p}(t_0)$ and
$\lambda_{1,p}(f(t_1),t_1)\geq\lambda_{1,p}(t_1)$. Hence we have
\begin{equation*}
\begin{aligned}
\ln&\left[\lambda_{1,p}(t_0)\cdot\left(\frac{\rho_0}{r}
-\frac{\rho_0}{r}e^{rt_0}+e^{rt_0}\right)\right]+\left(1-\frac
p2\right)\frac{C}{r}e^{rt_0}\\
&\geq\ln\left[\lambda_{1,p}(t_1)\cdot\left(\frac{\rho_0}{r}
-\frac{\rho_0}{r}e^{rt_1}+e^{rt_1}\right)\right]+\left(1-\frac
p2\right)\frac{C}{r}e^{rt_1}
\end{aligned}
\end{equation*}
for any $t_1<t_0$ sufficiently close to $t_0$. Since $t_0$ is
arbitrary, the result follows.

Case 2: $\chi(M^2)=0$.

Using $-\frac{C}{1+Ct}\leq R\leq C$ of Proposition \ref{procurest},
we have
\begin{equation}
\begin{aligned}\label{kddnmpj}
\frac{d}{dt}\lambda_{1,p}(f,t)\Big|_{t=t_0}&=\lambda_{1,p}(f(t_0),t_0)\int_M|
f|^p R d\mu+\left(\frac p2-1\right)\int_M|df|^pR d\mu\\
&\geq\lambda_{1,p}(f(t_0),t_0)\left[-\frac{C}{1+Ct_0}+\left(\frac
p2-1\right)C\right]
\end{aligned}
\end{equation}
where $f$ is defined by Lemma \ref{genlemma}. Then using similar
arguments above, we can obtain the desired results.

Case 3: $\chi(M^2)>0$.

Using $-Ce^{rt}\leq R\leq r+Ce^{rt}$ of Proposition \ref{procurest},
we get
\begin{equation}
\begin{aligned}\label{kkddnpde34}
\frac{d}{dt}\lambda_{1,p}(f,t)\Big|_{t=t_0}&=\lambda_{1,p}(f(t_0),t_0)\int_M|
f|^p R d\mu+\left(\frac p2-1\right)\int_M|df|^pR d\mu\\
&\,\,\,\,\,\,-\frac p2 r\lambda_{1,p}(f(t_0),t_0)\\
&\geq\lambda_{1,p}(f(t_0),t_0)\left[-r+\left(\frac
p2-2\right)Ce^{rt_0}\right]
\end{aligned}
\end{equation}
where $f$ is defined by Lemma \ref{genlemma}. Then using the
standard discussions above, we can obtain the desired results.
\end{proof}

\vspace{0.5em}

In the following we will finish the proof Theorem \ref{thm19}.
\begin{proof}[Proof of Theorem \ref{thm19}]
\emph{Step 1}: we first prove the case $p\geq 2$.

The case $\chi(M^2)=0$.

By Proposition \ref{procurest}, we have $R(x,t)\leq C$. Substituting
this into formula (\ref{kkddn2}),
\begin{equation}\label{case2}
\frac{d}{dt}\lambda_{1,p}(f,t)\Big|_{t{=}t_0}\leq\frac p2\cdot
C\lambda_{1,p}(f(t_0),t_0).
\end{equation}
Since $\lambda_{1,p}(f,t)$ is a smooth function with respect to
$t$-variable, we have
\begin{equation}\label{case2a}
\frac{d}{dt}\lambda_{1,p}(f,t)<\frac p2
\left(C+\epsilon\right)\lambda_{1,p}(f(t),t).
\end{equation}
for $\epsilon>0$ sufficiently small  in any sufficiently small
neighborhood of $t_0$. Integrating the above inequality with respect
to time $t$ on a sufficiently small time interval $[t_0,t_1]$, we
get
\begin{equation}\label{integcase1}
\ln\lambda_{1,p}(f(t_1),t_1)-\ln\lambda_{1,p}(f(t_0),t_0)<\frac p2
\left(C+\epsilon\right)t_1-\frac p2\left(C+\epsilon\right)t_0
\end{equation}
for any $t_1>t_0$ sufficiently close to $t_0$. Note that
$\lambda_{1,p}(f(t_0),t_0)=\lambda_{1,p}(t_0)$ and
$\lambda_{1,p}(f(t_1),t_1)\geq\lambda_{1,p}(t_1)$. So we have
\[
\ln\lambda_{1,p}(t_1)-\frac p2
\left(C+\epsilon\right)t_1<\ln\lambda_{1,p}(t_0)-\frac
p2\left(C+\epsilon\right)t_0
\]
for any $t_1>t_0$ sufficiently close to $t_0$. Since $t_0$ is
arbitrary, taking $\epsilon\rightarrow 0$, the result follows in the
case of $\chi=0$.

The case $\chi(M^2)\neq0$.

The method of the proof is similar to the case of $\chi(M^2)\neq0$.
Here we only give some key inequalities. Using $R\leq r+Ce^{rt}$ of
Proposition \ref{procurest} and formula (\ref{kkddn2}), we have
\begin{equation}
\begin{aligned}\label{kkddb4}
\frac{d}{dt}\lambda_{1,p}(f,t)\Big|_{t=t_0}\leq\frac
p2Ce^{rt_0}\lambda_{1,p}(f(t_0),t_0)
\end{aligned}
\end{equation}
where $f$ is defined by Lemma \ref{genlemma}. By similar arguments
the results follows.

\vspace{0.5em}

\emph{Step 2}: we consider the case $1<p<2$. Similarly, we only give
some key computations.

Case 1: $\chi(M^2)<0$.

Substituting (\ref{RRRk}) and $R\leq r+Ce^{rt}$ of Proposition
\ref{procurest} into formula (\ref{kkddn2}),
\begin{equation}\label{khhn2oy}
\frac{d}{dt}\lambda_{1,p}(f,t)\Big|_{t=t_0}\leq\lambda_{1,p}(f(t_0),t_0)
\left[\left(\frac p2-1\right)\cdot
\left(\frac{r}{1-(1-\frac{r}{\rho_0})e^{rt_0}}-r\right)
+Ce^{rt_0}\right]
\end{equation}
where $f$ is defined by Lemma \ref{genlemma}. Then using the
standard discussion as the case $\chi(M^2)=0$, we can obtain the
desired results.

Case 2: $\chi(M^2)=0$.

Substituting $-\frac{C}{1+Ct}\leq R\leq C$ of Proposition
\ref{procurest} into formula (\ref{kkddn2}), we have
\begin{equation}\label{casep12}
\frac{d}{dt}\lambda_{1,p}(f,t)\Big|_{t=t_0}\leq\lambda_{1,p}(f(t_0),t_0)
\left[\left(1-\frac p2\right)\cdot\frac{C}{1+Ct_0}+C\right]
\end{equation}
where $f$ is defined by Lemma \ref{genlemma}. Using similar
discussion above, the result follows.

Case 3: $\chi(M^2)>0$.

Using $-Ce^{rt}\leq R\leq r+Ce^{rt}$ of Proposition \ref{procurest},
we obtain
\begin{equation}\label{caseddo12}
\frac{d}{dt}\lambda_{1,p}(f,t)\Big|_{t=t_0}\leq\lambda_{1,p}(f(t_0),t_0)
\left[\left(1-\frac p2\right)\cdot r+\left(2-\frac
p2\right)Ce^{rt_0}\right]
\end{equation}
where $f$ is defined by Lemma \ref{genlemma}. Then the desired
results follow by the above similar discussions.
\end{proof}

\vspace{0.5em}

We should point out that for closed $2$-surfaces, we also have the
differentiability result along the unnormalized Ricci flow without
any curvature assumption.
\begin{corollary}\label{T103ab}
Let $g(t)$ and $\lambda_{1,p}(t)$ be the same as in Theorem
\ref{T101a}, where $n=2$. Then $\lambda_{1,p}(t)$ is differentiable
almost everywhere along the unnormalized Ricci flow.
\end{corollary}
\begin{proof}
For closed 2-surfaces, we know that the first eigenvalue of the
$p$-Laplace operator is differentiable almost everywhere along the
normalized Ricci flow. Hence the conclusion follows from the same
argument as in the proof of Theorem \ref{T100b}.
\end{proof}

\section{$p$-eigenvalue comparison-type theorem}\label{comparison}
In Riemannian geometry, a convenient way of understanding a general
Riemannian manifold is by comparison theorems. And many comparison
theorems have been obtained, such as the Hessian comparison theorem,
the Laplace comparison theorem, the volume comparison theorem, etc..

In this section, we will give another interesting comparison-type
theorem on a closed surface with the Euler characteristic
$\chi(M^2)< 0$, which is motivated by the work of J. Ling
\cite{Ling}. However, our proof may be different from Ling's.
Because we do not know the eigenvalue or eigenfunction
differentiability under the Ricci flow. Fortunately we can follow
similar arguments above and obtain our desired result.

Let $(M^2,g)$ be a closed surface. Let $K_g$, $\kappa_g$, $\mathrm
{Area}_g(M^2)$ denote the Gauss curvature, the minimum of the Gauss
curvature, the area of the surface, respectively. $\lambda_{1,p}(g)$
denotes the first eigenvalue of the $p$-Laplace operator $(p\geq2)$
with respect to the metric $g$. We now prove the comparison-type
theorem for $\lambda_{1,p}(g)$ on a closed surface with its Euler
characteristic is negative.

\begin{proof}[Proof of Theorem \ref{T201}]
Let $g(t)$ be the solution of the normalized Ricci flow on a closed
surface
\begin{equation}
\begin{aligned}\label{frac15}
\frac{\partial g(t)}{\partial t}=(r-R)g(t)
\end{aligned}
\end{equation}
with the initial condition $g(0)=g$, where $R$ is the scalar
curvature of the metric $g(t)$ and $r= {\int_{M^2}
Rd\mu}\big/{\int_{M^2} d\mu}$, which keeps the area of the surface
constant. In fact, from ({\ref{frac15}}) we have
\[
\frac{d}{dt}(d\mu)=(r-R)d\mu
\]
and
\[
\frac{d}{dt}\mathrm {Area}_{g(t)}(M^2)=\frac{d}{dt}\int_{M^2}
d\mu=\int_{M^2}(r-R)d\mu=0.
\]
Set $A:=\mathrm {Area}_{g(t)}(M^2)=\mathrm {Area}_{g}(M^2)$.
Obviously, along the normalized Ricci flow, the area $A$ remains
constant independent of time. By the Gauss-Bonnet theorem, $r$ is
determined by the Euler characteristic $\chi(M^2)$, i.e., $r=4\pi
\chi(M^2)/A<0$. So we know that $r$ is a negative constant and the
lower bounds of the scalar curvature $R$ are also negative.
Meanwhile, according to Theorem E in introduction, the metric $g(t)$
converges to a smooth metric $\bar{g}(=g(\infty))$ of constant Gauss
curvature $r/2$.

Note that $R/2$ is the Gauss curvature $K$ of the metric $g(t)$. Let
$\rho_0<0$ be the minimum of $R(0)$, i.e.,
\[
R(0)=2K(0)\geq \rho_0.
\]
Since $\chi(M^2)<0$, by Theorem \ref{coro17}, we know that
\begin{equation}\label{conv1}
\lambda_{1,p}(t)\cdot\left[\frac{\rho_0}{r}
-\frac{\rho_0}{r}e^{rt}+e^{rt}\right]^{p/2}
\end{equation}
is increasing along the normalized Ricci flow on $[0,\infty)$, where
$\rho_0=\inf\limits_{M^2} R(0)$.

Since that $r<0$ and $p\geq 2$, taking $t\rightarrow\infty$ in
(\ref{conv1}) and noticing that $\lambda_{1,p}(t)$ is continuous, we
conclude that
\[
\lambda_{1,p}(\infty)\geq\lambda_{1,p}(0)\cdot\left(\frac
{r}{\rho_0}\right)^{p/2}.
\]
Note that the metric $\bar{g}(=g(\infty))$ has constant Gauss
curvature $r/2$. So we have $\kappa_{\bar{g}}=r/2$. By the
definition for $\rho_0$, we also have $\rho_0=2\kappa_g$. Therefore
we conclude the following inequality
\[
\frac{\lambda_{1,p}(\bar{g})}{\lambda_{1,p}(g)}\geq\left(\frac
{\kappa_{\bar{g}}}{\kappa_g}\right)^{p/2}.
\]
This completes the proof of this theorem.
\end{proof}

\begin{remark}
(1). By Theorem \ref{coro17} and Theorem \ref{thm19}, using the same
method above, if $\chi(M^2)<0$ , we can also get some rough
estimates
\[
\frac{\lambda_{1,p}(\bar{g})}{\lambda_{1,p}(g)}\geq\exp\left[\left(1-\frac
p2\right)\frac Cr\right]\cdot\frac {\kappa_{\bar{g}}}{\kappa_g}\quad
(1<p<2);
\]
and
\[
\frac{\lambda_{1,p}(\bar{g})}{\lambda_{1,p}(g)}\leq e^{-\frac
Cr}\cdot\left(\frac {\kappa_{\bar{g}}}{\kappa_g}\right)^{\frac
p2-1}\,\,\,(1<p<2),\quad\quad\quad
\frac{\lambda_{1,p}(\bar{g})}{\lambda_{1,p}(g)}\leq\exp\left(-\frac
p2\cdot\frac Cr\right)\,\,\, (p\geq 2),
\]
where $C>0$ is a constant depending only on the metric $g$ and
$r=2\kappa_{\bar{g}}$.

(2). It would be interesting to find out if there exists a similar
comparison-type result for high dimensional closed manifolds. It
seems to be difficult to deal with the high-dimensional case. On the
other hand, can one have a similar result as theorem \ref{T201} if
one removes the condition: $\chi(M^2)<0$?

(3). Though we do not follow J. Ling's proof, the idea of proof
partly belongs to his. When $p=2$, our result reduces to J. Ling's
(see \cite{Ling}, Theorem 1.1).
\end{remark}

\section{First $p$-eigenvalue along general evolving metrics}\label{genmetr}
Following similar arguments in the proof of Theorem \ref{T101a}, in
this section, we discuss the monotonicity and differentiability for
the first eigenvalue of the $p$-Laplace with respect to general
evolving Riemannian metrics.

Let $(M^n, g(t))$ be a smooth one-parameter family of compact
Riemannian manifolds without boundary evolving for $t\in[0, T)$ by
\begin{equation}\label{genflow2}
\frac{\partial}{\partial t}g_{ij}=-2h_{ij}
\end{equation}
with $g(0)=g_0$. Let $H:=\mathrm{tr}\,h=g^{ij}h_{ij}$.

\vspace{0.5em}

We first have a analog of Proposition \ref{propo1} in Section
\ref{sec3}.
\begin{proposition}\label{genpropo1}
Let $g(t)$, $t\in[0,T)$, be a smooth family of complete Riemannian
metrics on a closed manifold $M^n$ satisfying (\ref{genflow2}) and
let $\lambda_{1,p}(t)$ be the first eigenvalue of the $p$-Laplace
operator $(p>1)$ under the evolving metrics (\ref{genflow2}). For
any $t_1, t_2 \in[0,T)$ with $t_2\geq t_1$, we have
\begin{equation}\label{genfrded}
\lambda_{1,p}(t_2)\geq \lambda_{1,p}(t_1)+\int^{t_2}_{t_1} \mathcal
{L}(g(\xi),f(\xi))d\xi,
\end{equation}
where
\begin{equation}\label{genfracww}
\mathcal{L}(g(t),f(t)):=p\int_M|df|^{p-2}h(\nabla f,\nabla f)d\mu-
p\int_M\Delta_p f\frac{\partial f}{\partial t}d\mu-\int_M
|df|^pHd\mu
\end{equation}
and where $f(t)$ is any $C^{\infty}$ function satisfying the
restrictions $\int_M|f(t)|^p d\mu_{g(t)}=1$ and $\int_M
|f(t)|^{p-2}f(t)d\mu_{g(t)}=0$, such that at time $t_2$, $f(t_2)$ is
the corresponding eigenfunction of $\lambda_{1,p}(t_2)$.
\end{proposition}
\begin{proof}
The proof is by straightforward computation, which is similar to the
proof of Proposition \ref{propo1}. Here we omit those details.
\end{proof}

Using this proposition, we have
\begin{theorem}\label{genT101}
Let $g(t)$ and $\lambda_{1,p}(t)$ be the same as Proposition
\ref{genpropo1}. If there exists a nonnegative constant $\epsilon$
such that
\begin{equation}\label{genassu1}
h_{ij}-\tfrac Hp g_{ij}\geq -\epsilon g_{ij}\quad
\quad\mathrm{in}\quad M\times[0,T)
\end{equation}
and
\begin{equation}\label{genassu2}
H>p\cdot\epsilon\quad \quad\mathrm{in}\quad M\times[0,T),
\end{equation}
then $\lambda_{1,p}(t)$ is strictly increasing and therefore
differentiable almost everywhere along the evolving Riemannian
metrics (\ref{genflow2}) on $[0,T)$.
\end{theorem}

\begin{proof}
This proof is similar to that of the previous theorems.
\end{proof}

\begin{remark}
(1). Assumptions (\ref{genassu1}) and (\ref{genassu2}) may not be
valid sometimes for some special curvature flow. For example, for
the normalized Ricci flow, the assumptions (\ref{genassu1}) and
(\ref{genassu2}) are not hold in general.

(2). This theorem may be compared to Theorem \ref{T101a} of this
paper. In fact, let $(M^n, g(t))$ be a complete solution of the
unnormalized  Ricci flow on $[0,T)$. This corresponds to
$h_{ij}=R_{ij}$ and $H=R$ in Theorem \ref{genT101}.
\end{remark}

\vspace{0.5em}

In the following, a general version of Lemma \ref{genlemma} is
stated as follows.
\begin{lemma}\label{genlemma2}
If $\lambda_{1,p}(t)$ is the first eigenvalue of
$\Delta_{p_{g(t)}}$, whose metric satisfying equation
(\ref{genflow2}) and $f(t_0)$ is the corresponding eigenfunction of
$\lambda_{1,p}(t)$ at time $t_0$, then we have
\begin{equation}\label{dd5mn}
\frac{d}{dt}\lambda_{1,p}(f,t)\Big|_{t=t_0}
=\lambda_{1,p}(f(t_0),t_0)\int_M|f|^pHd\mu
+\int_M|df|^{p-2}(ph_{ij}{-}Hg_{ij})f_i f_j d\mu,
\end{equation}
where $f(t)$ is any $C^{\infty}$ function satisfying the
restrictions $\int_M|f(t)|^p d\mu_{g(t)}=1$ and $\int_M
|f(t)|^{p-2}f(t)d\mu_{g(t)}=0$, such that at time $t_0$, $f(t_0)$ is
the corresponding eigenfunction of $\lambda_{1,p}(t_0)$.
\end{lemma}

In the same way as before, we can use this lemma to construct some
monotonic quantities about the first eigenvalue of the $p$-Laplace
operator along general evolving Riemannian metrics under some
curvature assumptions.

\vspace{1em}

Next we turn to study a particular geometric flow, i.e., Yamabe
flow. We will apply Theorem \ref{genT101} and Lemma \ref{genlemma2}
to the Yamabe flow. When $p=2$, the first author in \cite{Wu}
obtained some interesting results. The Yamabe flow was still
introduced by R.S. Hamilton, which is defined by
\begin{equation}
\begin{aligned}\label{Ypad}
\frac{\partial }{\partial t}g(x,t)&=-R(x,t)g(x,t),\\
g(x,0)&=g_0(x)
\end{aligned}
\end{equation}
where $R$ denotes the scalar curvature of $g(t)$. The normalized
Yamabe flow is defined by
\begin{equation}
\begin{aligned}\label{defL}
\frac{\partial }{\partial t}g(x,t)&=\left(r(t)-R(x,t)\right)g(x,t),\\
g(x,0)&=g_0(x)
\end{aligned}
\end{equation}
where $r(t):=\int_M R d\mu \big/\int_M d\mu$ is the average scalar
curvature of the metric $g(t)$.

For the unnormalized Yamabe flow, we have the following proposition.
\begin{proposition}\label{genYam1}
In Proposition \ref{genpropo1}, we replace general evolving metrics
by the unnormalized Yamabe flow (\ref{Ypad}). Then for any
$t_1,t_2\in[0,T)$ with $t_2\geq t_1$,
\begin{equation}\label{genYa1}
\lambda_{1,p}(t_2)\geq \lambda_{1,p}(t_1)+\int^{t_2}_{t_1} \mathcal
{L}(g(\xi),f(\xi))d\xi,
\end{equation}
where
\begin{equation}\label{genYam2}
\mathcal{L}(g(t),f(t)):=\frac{p-n}{2}\int_M|df|^p Rd\mu-p\int_M
\Delta_p f\frac{\partial f}{\partial t} d\mu.
\end{equation}
\end{proposition}

\begin{proof}
Substituting $h_{ij}=\tfrac R2 g_{ij}$ into Proposition
\ref{genpropo1}, the result follows.
\end{proof}

Using this proposition, we have
\begin{theorem}\label{ThmYam}
Let $g(t)$ and $\lambda_{1,p}(t)$ be the same as Proposition
\ref{genYam1}, where we assume $p\geq n$. If
\begin{equation}
R\geq0\quad \mathrm{and}\quad R\not\equiv0\quad
\quad\mathrm{in}\quad M^n\times\{0\},
\end{equation}
then $\lambda_{1,p}(t)$ is strictly increasing and therefore
differentiable almost everywhere along the unnormalized Yamabe flow
(\ref{Ypad}) on $[0,T)$.
\end{theorem}

\begin{proof}[Proof of Theorem \ref{ThmYam}]
Using basically the same trick as in proving Theorem \ref{T101a}, we
shall prove this result. Under the Yamabe flow (\ref{Ypad}), from
the constraint condition
\[
\frac{d}{dt}\int_M \left|f(t)\right|^p d\mu_{g(t)}=0,
\]
we have
\begin{equation}\label{xianzYam}
p\int_M |f|^{p-2}f \frac{\partial f}{\partial t} d\mu=\frac
n2\int_M|f|^pRd\mu.
\end{equation}
Note that at time $t_2$, $f(t_2)$ is the eigenfunction for the first
eigenvalue $\lambda_{1,p}(t_2)$ of $\Delta_{p_{g(t_2)}}$. Therefore
at time $t_2$, we have
\begin{equation}\label{xianzYamp}
\Delta_p f(t_2)=-\lambda_{1,p}(t_2)|f(t_2)|^{p-2}f(t_2).
\end{equation}
By Proposition \ref{genYam1}, at time $t_2$, we have
\begin{equation}
\begin{aligned}\label{henden2}
\mathcal{L}(g(t_2),f(t_2))&=\frac{p-n}{2}\int_M|df|^p Rd\mu-p\int_M
\Delta_p f\frac{\partial f}{\partial t} d\mu\\
&=\frac{p-n}{2}\int_M| df|^pRd\mu
+p\lambda_{1,p}(t_2)\int_M|f|^{p-2}f
\frac{\partial f}{\partial t}d\mu\\
&=\frac{p-n}{2}\int_M|df|^p R d\mu +\frac
n2\lambda_{1,p}(t_2)\int_M|f|^p R d\mu,
\end{aligned}
\end{equation}
where we used (\ref{xianzYamp}) and (\ref{xianzYam}). Notice that
the evolution of the scalar curvature $R$ under the Yamabe flow
(\ref{Ypad}) (see \cite{ChowYam}) is
\begin{equation}\label{dimscaYam}
\frac{\partial}{\partial t}R=(n-1)\Delta R+R^2.
\end{equation}
Applying the strong maximum principle, $R(g(0))\geq0$ and
$R(x_0,0)>0$ for some $x_0\in M^n$ imply that $R(x,t)>0$ for all
$(x,t)\in M^2\times(0,T)$. Since $p\geq n$, from (\ref{henden2}), we
then have $\mathcal{L}(g(t_2),f(t_2))>0$. Then using the same
arguments in proving Theorem \ref{T101a} yields the desired result.
\end{proof}

For the normalized Yamabe flow, we have
\begin{lemma}\label{genlemma3}
If $\lambda_{1,p}(t)$ is the first eigenvalue of
$\Delta_{p_{g(t)}}$, whose metric satisfying normalized Yamabe flow
(\ref{defL}) and $f(t_0)$ is the corresponding eigenfunction of
$\lambda_{1,p}(t)$ at time $t_0$, then we have
\begin{equation}\label{ppk5}
\frac{d}{dt}\lambda_{1,p}(f,t)\Big|_{t=t_0}=\frac{p-n}{2}\int_M|df|^p
(R-r)d\mu+\frac n2\lambda_{1,p}(f(t_0),t_0)\int_M|f|^p(R-r)d\mu,
\end{equation}
\end{lemma}
\begin{proof}
Substituting $h_{ij}=\frac{R-r}{2} g_{ij}$ into Lemma
\ref{genlemma2}, then the result follows.
\end{proof}

In the end of this section, we will apply Lemma \ref{genlemma3} to
construct some monotonic quantities along the unnormalized Yamabe
flow, generalizing earlier results for $p=2$ derived by the first
author in \cite{Wu}.

\begin{theorem}\label{Yamadiff}
Let $g(t)$, $t\in[0,T)$, be a solution of the unnormalized Yamabe
flow (\ref{Ypad}) on a closed manifold $M^n$ and let
$\lambda_{1,p}(t)$ be the first eigenvalue of the $p$-Laplace
operator  of the metric $g(t)$. Assume that the initial scalar
curvature $R(g(0))>0$. Then on one hand, if $1<p<n$,
\begin{equation}
\lambda_{1,p}(t)\cdot\left(1-\rho_0t\right)^{n/2}
\cdot\left(1-\sigma_0t\right)^{\frac{p-n}{2}}
\end{equation}
is increasing along the unnormalized Yamabe flow on $[0,T'')$ and if
$p\geq n$,
\begin{equation}
\lambda_{1,p}(t)\cdot\left(1-\rho_0t\right)^{p/2}
\end{equation}
is increasing along the unnormalized Yamabe flow on $[0,T')$. On the
other hand, the following quantities
\begin{equation}\label{Yacase3}
\lambda_{1,p}(t)\cdot\left(1-\rho_0t\right)^{\frac{p-n}{2}}
\cdot\left(1-\sigma_0t\right)^{n/2}\quad\quad (1<p<n)
\end{equation}
and
\begin{equation}
\lambda_{1,p}(t)\cdot\left(1-\sigma_0t\right)^{p/2}
\quad\quad\quad\quad\quad\quad(p\geq n)
\end{equation}
are both decreasing along the unnormalized Yamabe flow on $[0,T'')$,
where $\rho_0:=\inf_{M^2}R(0)$, $\sigma_0:=\sup_{M^2}R(0)$,
$T':=\min\{\rho_0^{-1},T\}$ and $T'':=\min\{\sigma_0^{-1},T\}$.
Therefore $\lambda_{1,p}(t)$ is differentiable almost everywhere
along the unnormalized Yamabe flow.
\end{theorem}

\begin{proof}
Since this proof is similar to the proofs of Theorems \ref{coro17}
and \ref{thm19}, we only give some key inequalities. Note that under
the unnormalized Yamabe flow,
\begin{equation*}
\frac{\partial}{\partial t}R=(n-1)\Delta R+R^2.
\end{equation*}
Applying the maximum principle to this equation, we have lower and
upper bounds of the scalar curvature $R$
\begin{equation}\label{yameequ1}
R(x,t)\geq\frac{\rho_0}{1-\rho_0t},\quad t\in[0,T');\quad \quad
R(x,t)\leq\frac{\sigma_0}{1-\sigma_0t},\quad t\in[0,T'').
\end{equation}
where $\rho_0:=\inf_{M^n}R(0)$, $\sigma_0:=\sup_{M^n}R(0)$,
$T':=\min\{\rho_0^{-1},T\}$ and $T'':=\min\{\sigma_0^{-1},T\}$.

By (\ref{ppk5}) of Lemma \ref{genlemma3}, we also have
\begin{equation}\label{ppw1}
\frac{d}{dt}\lambda_{1,p}(f,t)\Big|_{t=t_0}=\frac{p-n}{2}\int_M|df|^p
Rd\mu+\frac n2\lambda_{1,p}(f(t_0),t_0)\int_M|f|^pRd\mu,
\end{equation}
where $f$ is defined by Lemma \ref{genlemma3}.

\vspace{0.5em}

On one hand, if $1<p<n$, by (\ref{yameequ1}) and (\ref{ppw1}) we
conclude
\[
\frac{d}{dt}\lambda_{1,p}(f,t)\Big|_{t=t_0}
\geq\lambda_{1,p}(f(t_0),t_0)\left[\frac{p-n}{2}
\cdot\frac{\sigma_0}{1-\sigma_0t_0}+\frac
n2\cdot\frac{\rho_0}{1-\rho_0t_0}\right].
\]
Then following the exactly same arguments as in proving Theorem
\ref{coro17}, we see that
\[
\lambda_{1,p}(t)\cdot\left(1-\rho_0t\right)^{n/2}
\cdot\left(1-\sigma_0t\right)^{\frac{p-n}{2}}
\]
is increasing along the unnormalized Yamabe flow on $[0,T'')$.

If $p\geq n$, by (\ref{yameequ1}) and (\ref{ppw1}) we have
\[
\frac{d}{dt}\lambda_{1,p}(f,t)\Big|_{t=t_0}\geq\frac
p2\lambda_{1,p}(f(t_0),t_0)\cdot\frac{\rho_0}{1-\rho_0t_0}.
\]
Then using our standard arguments,  we conclude that
\[
\lambda_{1,p}(t)\cdot\left(1-\rho_0t\right)^{p/2}
\]
is increasing along the unnormalized Yamabe flow on $[0,T')$.

\vspace{0.5em}

On the other hand, we consider the decreasing quantities under the
unnormalized Yamabe flow. If $1<p<n$, by (\ref{yameequ1}) and
(\ref{ppw1}), we can get
\[
\frac{d}{dt}\lambda_{1,p}(f,t)\Big|_{t=t_0}
\leq\lambda_{1,p}(f(t_0),t_0)\left[\frac{p-n}{2}
\cdot\frac{\rho_0}{1-\rho_0t_0}+\frac
n2\cdot\frac{\sigma_0}{1-\sigma_0t_0}\right].
\]
Using the same arguments as in proving Theorem \ref{thm19}, then
(\ref{Yacase3}) follows.

If $p\geq n$, by (\ref{yameequ1}) and (\ref{ppw1}), we can obatin
\[
\frac{d}{dt}\lambda_{1,p}(f,t)\Big|_{t=t_0}\leq\frac
p2\lambda_{1,p}(f(t_0),t_0)\cdot\frac{\sigma_0}{1-\sigma_0t_0}.
\]
By the standard arguments of Theorem \ref{thm19}, we conclude that
\[
\lambda_{1,p}(t)\cdot\left(1-\sigma_0t\right)^{p/2}
\]
is decreasing along the unnormalized Yamabe flow on $[0,T'')$.
\end{proof}

\section*{Acknowledgment}
The authors would like to thank the referee for helpful comments and
suggestions to improve this paper.

\end{document}